\documentclass{article}
\usepackage[utf8]{inputenc}
\usepackage[left=2cm,right=2cm,top=2cm,bottom=2cm]{geometry}
\pdfoutput=1
\usepackage{lineno,hyperref}
\modulolinenumbers[5]
\usepackage{amsmath}
\usepackage{amssymb}
\usepackage{float}
\usepackage{mathtools}
\usepackage{lipsum}
\usepackage{caption}
\usepackage{subcaption}
\usepackage{commath}
\usepackage{cases}
\usepackage{amsthm}
\usepackage{graphicx}
\usepackage{listings}
\usepackage{multicol}
\usepackage{algorithm}
\usepackage{algpseudocode}

\usepackage[sc,osf]{mathpazo}
 \newcommand{\spann}{\textrm{span}}

\newcommand{\uu}{\mathbf{u}}

\newcommand{\xx}{\textbf{x}}

\newcommand{\vv}{v}

\newcommand{\mm}{\mathcal{M}}

\newcommand{\wt}{\widetilde{t}}

\newcommand{\pt}{\partial}

\newtheorem{theorem}{Theorem}[section]

\newtheorem{corollary}[theorem]{Corollary}

\newtheorem{remark}[theorem]{Remark}

\begin{document}
 \title{\noindent Error estimate of the Non-Intrusive Reduced Basis (NIRB) two-grid method with parabolic equations}
\maketitle
\normalsize
\begin{center}
\author{Elise Grosjean \footnotemark[1],}
\author{Yvon Maday \footnotemark[2] \footnotemark[3]}
\end{center}

\footnotetext[1]{Felix-Klein-Institut für Mathematik, Kaiserslautern TU, 67657, Deutschland}
\footnotetext[2]{Sorbonne Universit\'e and Universit\'e de Paris Cité, CNRS, Laboratoire Jacques-Louis Lions (LJLL), F-75005 Paris, France}
\footnotetext[3]{Institut Universitaire de France}

\date

\begin{abstract}
Reduced Basis Methods (RBMs) are frequently proposed to approximate parametric problem solutions. 
They can be used to calculate solutions for a large number of parameter values (e.g. for parameter fitting) as well as to approximate a solution for a new parameter value (e.g. real time approximation with a very high accuracy).
They intend to reduce the computational costs of High Fidelity (HF) codes. They necessitate well-chosen solutions, called snapshots, that have been previously computed (e.g. offline) with a HF classical method, involving, for instance a fine mesh (finite element or finite volume) and generally require a profound modification of the HF code, in order for the online computation to be performed in short (or even real) time. \\
We will focus on the Non-Intrusive Reduced Basis (NIRB) two-grid method. Its main advantage is that it uses the HF code exclusively as a "black-box," as opposed to other so-called intrusive methods that require code modification. This is very convenient when the HF code is a commercial one that has been purchased, as is frequently the case in the industry. The effectiveness of this method relies on its decomposition into two stages, one offline (classical in most RBMs as presented above) and one online. The offline part is time-consuming but it is only performed once.
On the contrary, the specificity of this NIRB approach is that, during the online part, it solves the parametric problem on a coarse mesh only and then improves its precision. As a result, it is significantly less expensive than a HF evaluation. This method has been originally developed for elliptic equations with finite elements and has since been extended to finite volume.\\
In this paper, we extend the NIRB two-grid method to parabolic equations.
We recover optimal estimates in $L^{\infty}(0,T;H^1(\Omega))$ using as a model problem, the heat equation. Then, we present numerical results on the heat equation and on the Brusselator problem.
\end{abstract}


\section{Introduction.}
\label{introduction}
Let $\Omega$ be a bounded domain in $\mathbb{R}^d$, with $d \leq 3$ and a smooth enough boundary $\pt \Omega$, and consider a parametric problem $\mathcal{P}$ on $\Omega$. 
Non-Intrusive Reduced Basis (NIRB) methods are an alternative to classical Reduced Basis Methods (RBMs) for approximating the solutions of such problems where the parameter is denoted as $\mu$, in a given set $\mathcal{G}$ \cite{chakirrect,inproceedings} (see also different NIRB methods \cite{Casenave2014,barrault2004empirical,willcox} from the two-grid method). From an engineering point of view, they may be more practical to implement than intrusive RBMs, as they only require the execution of the High-Fidelity (HF) code as a ``black-box'' solver.
The NIRB methods, like most RBMs, rely on the assumption that the manifold of all solutions $\mathcal{S} = \{u(\mu), \mu \in \mathcal{G} \}$ has a small Kolmogorov width \cite{kolmo} (in what follows, $u_h(\mu)$ will refer to the HF solution for the parameter $\mu$). Let us first recall the method for stationnary problem.
\subsection{Reminders on the NIRB two-grid method for stationnary problems.}
In the context of a finite element or finite volume HF solver, the two-grid method involves two partitioned meshes (or "grid"), one fine mesh $\mm_h$ and one coarse $\mm_H$, where the respective sizes $h$ and $H$ of the meshes are such that $h << H$. The size $h$ (respectively $H$) is defined
as \begin{equation}
  \label{meshsize}
  h = \underset{K\in \mm_h}{\textrm{max }} h_K \; (\textrm{respectively  }H = \underset{K\in \mm_H}{\textrm{max }} H_K),
\end{equation}
where the diameter $h_K$ (or $H_K$) of any element $K$ in a mesh is equal to $\underset{x,y \in K}{\sup \ } |x-y|, K \in M_h$ (or~$\in M_H)$.\\
The fine mesh is used to construct the Reduced Basis (RB).\\
 The reduced space $X_h^N:=~Span\{u_h(\mu_i)~|~i~=~1~, \dots N\}$ is generated using $N$ snapshots.
The solution for a new parameter is then roughly and quickly approximated using a coarse mesh.
The latter, as well as the algorithm's offline-online decomposition, are critical components in reducing complexity.
Below are the main steps of the NIRB two-grid algorithm:

\begin{itemize}
\item ``Offline stage'': \\

  First, in this stage, the RB functions that belong to the reduced space denoted $X_h^N$ are prepared on the fine mesh using a greedy procedure \cite{greedy,rb} (an alternative is to use a Proper Orthogonal Decomposition (POD) \cite{POD3,hesthaven2016certified}).
The greedy procedure computes the modes by iteratively selecting some suitable parameters $\mu_1, \dots, \mu_N \in \mathcal{G}$ and computing the approximate solutions $u_h(\mu_1), \dots,u_h(\mu_N)$.
This part is time-consuming, but it is only performed once, as with other RBMs.
After running a Gram-Schmidt orthonormalization algorithm, we obtain $N$ $L^2$-orthonormalized basis functions, denoted $(\Phi_i^h)_{i=1,\dots,N}$.
In order to improve the accuracy of the online reconstruction (as detailed in section \ref{NIRBproof}), we run the following eigenvalue problem:
    \begin{numcases}
      \strut \textrm{Find } \Phi^h \in X_h^N, \textrm{ and } \lambda \in \mathbb{R} \textrm{ such that: }    \nonumber\\
        \forall \vv \in X_h^N, \int_{\Omega} \nabla \Phi^h \cdot \nabla \vv \ d\xx= \lambda \int_{\Omega} \Phi^h \cdot \vv \ d\xx,
    \end{numcases}
    and we obtain an increasing sequence of eigenvalues $\lambda_i$, as well as orthogonal eigenfunctions $(\Phi_i^h)_{i=1,\cdots,N}$, orthonormalized in $L^2(\Omega)$ and orthogonal in $H^1(\Omega)$, and define a new basis of the space $X_h^N$.\\
    
    As written above, a coarse approximation for a new parameter $\mu \in \mathcal{G}$ will be used during the online stage. As we will see later, for any parameter  $\mu_k, \ k=1,\dots,N$, the classical NIRB approximation differs from the HF $u_h(\mu_k)$ computed in the offline stage. Thus, as proposed in \cite{madaychakir}, we use a "rectification post-processing" and introduce a rectification matrix, denoted $\mathbf{R}$ to cover this for these particular choices of $\mu=\mu_k,\ k=1,\dots,N$, and improve NIRB accuracy for other instance of $\mu$.
In addition to the fine snapshots, coarse snapshots are used in the construction of this matrix, which are generated using the same parameters as for the fine snapshots. Then, we compute the vectors
    \begin{equation}
              \mathbf{R}_i=(\mathbf{A}^T\mathbf{A}+\delta \mathbf{I}_{N})^{-1}\mathbf{A}^T \mathbf{B}_i, \quad  i=1, \cdots,N,
              \label{rectiffz1}
    \end{equation}
    where
    \begin{align}
      & \forall i=1,\cdots,N,\quad  \textrm{ and }  \quad \forall \mu_k \in  \mathcal{G},\nonumber \\
      & \quad  A_{k,i}=\int_{\Omega} \uu_H(\mu_k) \cdot \Phi_i^h\ \textrm{d}\xx, \label{Aj} \\
      & \quad B_{k,i}=\int_{\Omega} \uu_h(\mu_k) \cdot \Phi_i^h\ \textrm{d}\xx,     \label{Bj}
\end{align}
    where $\mathbf{I}_{N}$ refers to the identity matrix and $\delta$ is a regularization term, as proposed in \cite{chakirrect}.
    
  \item ``Online stage'':\\
    Then, for a new parameter $\mu \in \mathcal{G}$ for which we want to estimate the solution, a coarse approximation of the solution, denoted $u_H(\mu)$, is first computed "online."
This coarse approximation is, of course, not of sufficient precision, but it is calculated much faster than the HF one.
The NIRB post-processing then improves precision significantly by projecting $u_H(\mu)$ on the RB in a very short runtime \cite{chakirrect,NIRB1,VF,inproceedings}.
The classical NIRB approximation is given by
\begin{equation}
          \label{stationary}
       u_{Hh}^N(\mu):= \overset{N}{\underset{i=1}{\sum}}(u_H(\mu),\Phi_i^h)\ \Phi_i^h.
\end{equation}
where $(\cdot,\cdot)$ denotes the $L^2$-inner product.
      Now, to improve precision, we use the ``rectification post-treatment'', and the NIRB approximation reads 
      \begin{equation}
       \label{stationarywithrect}
       R[u_{Hh}^N](\mu):= \overset{N}{\underset{i,j=1}{\sum}}R_{ij}\ (u_H(\mu),\Phi_j^h)\ \Phi_i^h.
         \end{equation}
         Note that, when the relaxation parameter $\delta$ is equal to $0$ the rectification process allows to retrieve the fine coefficients (given by \eqref{Bj})  from the coarse ones (given by \eqref{Aj}) for the parameters $\mu=\mu_k$, $k=1,\dots,N$. In other words, with  $\delta=0$, we have
         \begin{equation*}
           R[u_{Hh}^N] (\mu_k)=u_h (\mu_k), \quad k=1, \dots , N.
           \end{equation*}
\end{itemize}

\subsection{Motivation and earlier works.}

The two-grid method is simple to implement and can be used for a variety of PDEs and approximations.
Furthermore, because it is non-intrusive, it is suitable for a wide range of problems.
To our knowledge, however, this method has not yet been studied or implemented in the context of time-dependent problems ~\cite{NIRB1,chakirrect,inproceedings,censity}.\\
  The two-grid method has been developed and analyzed for elliptic equations in the context of FEM (with Céa's and Aubin-Nitsche's lemmas) in \cite{madaychakir}. The energy-error estimate is then given by
    \begin{equation}
    \norm{u(\mu) - u_{Hh}^N(\mu)}_{H^1(\Omega)}\leq \varepsilon(N) +C_1 h + C_2(N) H^2, 
    \label{estimationNIRBFEM}
    \end{equation}
    where $C_1$ and $C_2$ are constants independent of $h$ and $H$, and $C_2$ depends on $N$ only. The term $\varepsilon(N)$ depends on a proper choice of the RB space as a surrogate for the best approximation space associated to the Kolmogorov $N$-width. 
    It decreases when $N$ increases and it is linked to the error between the fine solution and its projection on $X_h^N$, given by
    \begin{equation}
      \label{truerror}
  \norm{u_h(\mu) - \overset{N}{\underset{i=1}{\sum}}(u_h(\mu),\Phi_i^h)\ \Phi_i^h}_{H^1(\Omega)}.
    \end{equation}
    The second term in \eqref{estimationNIRBFEM}, $C_1 \ h $, is a contribution obtained through Céa's lemma for the RB elements and the second one, $C_2(N) \ H^2$, through Aubin-Nitsche's lemma for the coarse grid approximation of $u(\mu)$. Note that since the constant $C_2$ increases with $N$, a trade-off needs to be done between increasing $N$ to obtain a more accurate manifold, and keeping a constant $C_2$ as low as possible. \\
    The estimate \eqref{estimationNIRBFEM} proves that in the parabolic context with FEM, if the coarse mesh size is chosen so that $H^2=h$, we obtain the optimal $H^1$ convergence rate. Furthermore, it has been numerically shown that for a large range of $H$, the rectification post-treatment allows for the recovery of the fine solution's accuracy.\\
    This two-grid method has also been generalized and analyzed in the context of finite volume schemes such as \cite{VF}, in which a surrogate to Aubin-Nitsche's is used.
    
    \subsection{Outline  of  the  paper.}

This article is about the application of NIRB to time-dependent problems and its numerical analysis in the context of parabolic equations.\\
We will first define the NIRB approximation with and without the rectification post-treatment, as an extention of \eqref{stationary} and \eqref{stationarywithrect}.
We will then prove theoretically that we can recover optimal error estimates in $L^{\infty}(0,T;H^1(\Omega))$. The theorem  \ref{EllipticEst}  on the numerical analysis of the approach's convergence provides our main result. Then we will present numerical results \ref{results} with and without the rectification post-processing. We will illustrate that this post-treatment allows us to retrieve the fine accuracy in a parabolic context as well.\\

The remainder of this paper is structured as follows. The mathematical context is described in section  \ref{paragraphParabolic}.
The two-grids method is presented in section \ref{NIRB} in the context of parabolic equations. The proof of theorem \ref{EllipticEst} is covered in the section \ref{NIRBproof}.
Finally, the implementation is discussed in the last section \ref{results}, and the theoretical results are illustrated with numerical results on the NIRB method with and without the rectification post-treatment.\\

In the next sections, $C$ will denote various positive constants independent of the size of the meshes $h$ and $H$ and of the parameter $\mu$, and $C(\mu)$ will denote constants independent of the sizes of the meshes $h$ and $H$ but dependent of $\mu$. 

\section{Mathematical Background.}
\label{paragraphParabolic}

\subsection{The continuous problem.}
We will consider the following heat equation on the domain $\Omega$ with homogeneous Dirichlet conditions, which takes the form

\begin{numcases}{}
& $u_t-\mu \Delta u =f,\ \textrm{ in }\Omega\times ]0,T]$, \nonumber\\
& $u(\cdot,0)=u^0, \ \textrm{ in }\Omega$,   \label{heatEq2} \\
& $u(\cdot,t)=0, \ \textrm{ on }\pt \Omega \times ]0,T]$, \notag
\end{numcases} 
where $f \in L^2(\Omega \times [0,T] )$, while $u^0 \in H_0^1(\Omega)$ and $0 < \mu  \in \mathcal{G}$ is the parameter. For any $t>0$, the solution $u(\cdot,t) \in H_0^1(\Omega)$, and $u_t(\cdot, t) \in L^2(\Omega)$ stands for the derivative of $u$ with respect to time.
   
We use the conventional notations for space-time dependent Sobolev spaces \cite{lions1961problemes}
\begin{align*}
 & L^p(0,T;V):=\{u(\xx,t)\ | \ \norm{u}_{L^p(0,T;V)}:=\Big(\int_0^T \norm{u(\cdot,t)}_{V}^p \ dt \Big)^{1/p}< \infty \}, \ 1\leq  p < \infty,\\
  &    L^{\infty}(0,T;V):=\{u(\xx,t)\ | \ \norm{u}_{L^{\infty}(0,T;V)}:= ess \ \underset{0\leq t \leq T }{sup} \ \norm{u(\cdot, t)}_V< \infty \},\\
\end{align*}
where $V$ is a real Banach space with norm $\norm{\cdot }_V.$ 
The variational form of \eqref{heatEq2} is given by:\\
\begin{numcases}{}
  & Find $u\in L^2(0,T;H_0^1(\Omega))$ with $u_t \in L^2(0,T;H^{-1}(\Omega))$ such that \nonumber \\
 & $ (u_t(t,\cdot),v)+a(u(t,\cdot),v;\mu)=(f(t,\cdot),v), \ \forall v \in H_0^1(\Omega) \textrm{ and } t\in (0,T)$,  \label{varpara} \\
 & $ u(\cdot, 0)=u^0 , \textrm{ in } \Omega,$ \nonumber
\end{numcases}
where $a$ is given by \begin{equation}
  a(w,v;\mu)=\int_{\Omega}\mu \nabla w(\xx)\cdot \nabla v (\xx)\ d\xx,\quad \forall w, v\in H_0^1(\Omega).
  \label{ateerm1}
\end{equation}
We remind that \eqref{varpara} is well posed (see \cite{evans10} for the existence and the uniqueness of solutions to problem \eqref{varpara}) and we refer to the notations of \cite{evans10}.
\begin{remark}{(On the stability).}
We intend to state estimates for the NIRB approximation for all time snapshots, that is related to maximum-norm in time with the either $L^2$ norm or $H^1$ norm in space, i.e. in  $L^{\infty}(0,T;L^2(\Omega))$ and $L^{\infty}(0,T;H^1(\Omega))$, which is stronger than with the usual stability study of the parabolic equation \eqref{heatEq2}.
Let us remind the classical (or less standard) stability results. We derive from \eqref{varpara} by using $v=u$
\begin{equation}
\label{prout}
(u_t,u)+\mu \norm{\nabla u}^2  = |(f,u)|.
\end{equation}
From the Young and Poincaré inequalities, there exists $C>0$ such that
\begin{equation*}
|(f,u)| \leq (\frac{1}{2\mu} \norm{f}^2) +  \frac{\mu}{2} \norm{\nabla u}^2.
\end{equation*}
and since \begin{equation*}
(u_t,u)=\frac{1}{2}\frac{d}{dt} \norm{u}^2,
\end{equation*} \eqref{prout} yields
\begin{equation*}
\frac{d}{dt} \norm{u}^2 + \mu \norm{\nabla u}^2 \leq \frac{1}{\mu} \norm{f}^2,
\end{equation*}
and integrating over $(0,t)$ for all $t\leq T$,  we end up with
\begin{equation*}
\norm{u(t)}^2+ \mu \int_0^t \norm{\nabla u(s)}^2 \ ds \leq C(\norm{u^0}^2+\frac{1}{\mu}\int_{0}^t \norm{f(s)}^2 \ ds),
\end{equation*}
which gives
\begin{equation*}
\norm{ u}_{L^{\infty}(0,T;L^2(\Omega))}^2 + \mu \norm{ u}_{L^2(0,T;H_0^1(\Omega))}^2 \leq C (\norm{u^0}^2_{L^2(\Omega)}+\frac{1}{\mu}\norm{f}^2_{L^2(0,T;L^2(\Omega))}).
\label{est5ParabolBisL2}
\end{equation*}
That establishes the first stability result.\\
For the $L^{\infty}(0,T;\ H^1(\Omega))$ stability, a classical result is
\begin{equation*}
  \int_0^t \norm{u_t(s)}^2 \ ds+ \mu \norm{\nabla u(t)}^2  \leq \mu \ \norm{\nabla u^0}^2 +  \int_0^t \norm{f(s)}^2 \ ds,
\end{equation*}
and this ``a priori'' estimate then leads to the second stability results
\begin{equation*}
  \norm{ u}_{L^{\infty}(0,T;H_0^1(\Omega))}^2 \leq  \norm{\nabla u^0}^2_{L^2(\Omega)} + \frac{1}{\mu}  \norm{f(s)}^2_{L^2(0,T;L^2(\Omega))}.
\end{equation*}

\end{remark}
\subsection{The various discretizations.}
As in previous work on the NIRB FEM applied to elliptic equations \cite{madaychakir}, we consider one fine spatial grid for computing "offline" snapshots associated with few parameter values and one coarse grid for the coarse solution, with sizes denoted as $h$ and $H$ (with $h<< H$) \eqref{meshsize}. These grids are used for the spacial discretizations of the weak formulation of problem \eqref{heatEq2}.
We employed $\mathbb{P}_1$ finite elements to discretize in space, so let $V_h$ and $V_H$ be continuous piecewise linear finite element functions (on fine and coarse mesh, respectively) that vanish on the boundary $\pt \Omega$.
We consider the projection operator  $P^1_h$ on $V_h$ ($P^1_H$ on $V_H$ is defined similarly) which is given by
\begin{equation}
(\nabla P^1_h u, \nabla v)=(\nabla u, \nabla v),\quad  \forall v \in V_h,
  \label{P1opParabol}
\end{equation}

In the context of time-dependent problems, a time stepping method of finite difference type is used to get a fully discrete approximation of the solution of \eqref{heatEq2}. 
    We consider two different time grids:
\begin{itemize}
\item One time grid, denoted $F$, is employed for the fine solution (for the snapshots construction). To avoid making notations more cumbersome, we will consider a uniform time step $\Delta t_F$. The time levels can be written $t^n = n \ \Delta t_F$, where $n \in \mathbb{N}^*$.
\item Another time grid, denoted $G$, is used for the coarse solution. By analogy with the fine grid, we consider a uniform grid with time step $\Delta t_G$. Now, the time levels are written $\widetilde{t}^m = m \ \Delta t_G$, where $m \in \mathbb{N}^*$. 
\end{itemize}

As in the previous analysis with elliptic equations, the NIRB algorithm is designed to recover the optimal estimate in space. However, there is no such argument as the Aubin-Nitsche argument for time stepping methods, so we must consider time discretizations that provide the same precision with larger time steps. Thus, we consider a higher order time scheme for the coarse solution.
We will use an Euler scheme (first order approximation) for the fine solution and a Crank-Nicolson scheme (second order approximation) for the coarse solution with our model problem.\\
Thus, we deal with two kind of notations for the discretized solutions:
\begin{itemize}
\item $u_h(\xx,  t)$ and $u_H(\xx,t)$ that respectively denote the fine and coarse solutions of the spatially semi-discrete solution, at time $t \geq 0$.
\item $u_h^n(\xx)$ and $u_H^m(\xx)$ that respectively denote the fine and coarse full-discretized solutions at time $t^n=n \times \Delta t_F$ and $\widetilde{t}^m = m\times \Delta t_G$.
\end{itemize}

\begin{remark}
  To simplify the notations, we consider that both time grids end at time $T$ here,
  \begin{equation*}
    T\ =\ N_T\ \Delta t_F \ =\  M_T\ \Delta t_G.
  \end{equation*}     
\end{remark}

The semi-discrete form of the variational problem \eqref{varpara} writes for the fine mesh (similarly for the coarse mesh):
\begin{numcases}{}
  & Find $u_h(t)=u_h(\cdot,t) \in V_h$ for $t\in [0,T]$ such that \nonumber \\
 & $ (\pt_t u_h(t),v_h)+a(u_h(t),v_h;\mu)=(f(t),v_h), \ \forall v_h \in V_h \textrm{ and } t \in ]0,T]$, \label{varpara1disc} \\
 & $ u_h(\cdot, 0)\ =\ u_h^0 \ =\ P_h^1(u^0)$. \nonumber
\end{numcases}
The full discrete form of the variational problem \eqref{varpara} for the fine mesh with implicit Euler scheme writes:
\begin{numcases}{}
  &  Find $u_h^n \in V_h$ for $n= 0,\dots,N_T$ such that \nonumber \\
 & $ (\overline{\pt} u_h^n,v_h)+a(u_h^n,v_h;\mu)=(f(t^n),v_h), \ \forall v_h \in V_h \textrm{ and } n = 1,\dots,N_T$,\label{varpara2disc}\\
 & $ u_h(\cdot, 0)\ =\ u_h^0 $, \nonumber
\end{numcases}
where the time derivative in the variational form of the problem \eqref{varpara1disc} has been replaced by a backward difference quotient, $\overline{\pt} u^n_h=\frac{u^n_h-u^{n-1}_h}{\Delta t_F}$.\\
For the coarse mesh With Crank-Nicolson scheme, and with the notation $\overline{\pt} u^m_H=\frac{u^m_H-u^{m-1}_H}{\Delta t_G}$, it becomes:
\begin{numcases}{}
  &  Find $u_H^m \in V_H$ for $m = 0,\dots,M_T$, such that \nonumber \\
 & $ (\overline{\pt} u_H^m,v_H)+a(\frac{u_H^m + u_H^{m-1}}{2},v_H;\mu)=(f(\widetilde{t}^{m-\frac{1}{2}}),v_H), \ \forall v_H \in V_H \textrm{ and } m = 1,\dots M_T$, \nonumber \\
 & $ u_H(\cdot, 0)\ =\ u_H^0 $, \label{varpara2discCN}
\end{numcases}
where $\widetilde{t}^{m-\frac{1}{2}}= \frac{\widetilde{t}^{m}+ \widetilde{t}^{m-1}}{2}.$

Let us recall a few results from  \cite{thomee2}, on the FEM classical estimates and on both finite difference schemes used. These results will be useful for the proof of Theorem \ref{EllipticEst}.\\

The following estimates are well known to hold with a FEM semi-discretization in space:

\begin{theorem}[Corollary of Theorem 1.2 \cite{thomee2}]
\label{th12}
Let $\Omega$ be a convex polyhedron. Let $u \in W^{1,1}(0,T; H^2(\Omega))$ be the solution of \eqref{heatEq2} with $u^0 \in H^2(\Omega)$ and $u_h$ be the semidiscretized variational form \eqref{varpara1disc}. Then
\begin{equation}
\forall t \geq 0, \quad \norm{u(t)-u_h(t)}_{L^2(\Omega)} \leq C h^2 \Big[ \norm{u^0}_{H^2(\Omega)}+\int_0^t\norm{u_t}_{H^2(\Omega)}\ ds \Big].
\label{Th1Parabol}
\end{equation}
\end{theorem}
Once fully discretized on a fine mesh with the backward Euler Galerkin method, the estimate \eqref{Th1Parabol} is replaced by the estimate above.
\begin{theorem}[Corollary of Theorem 1.5 \cite{thomee2}]
\label{th15}
Let $\Omega$ be a convex polyhedron. Let $u \in W^{1,1}(0,T; H^2(\Omega)) \cap W^{2,1}(0,T; L^2(\Omega))$ be the solution of \eqref{heatEq2} with $u^0 \in H^2(\Omega)$, let $u^n_h$ be the solution of \eqref{varpara2disc}. If $\norm{u_h^0 - u^0}_{L^2(\Omega)} \leq C h^2 \norm{u^0}_{H^2(\Omega)}$, we have
\begin{equation}
\forall n =0,\dots,N_T \quad , \norm{u(t^n)-u^n_h}_{L^2(\Omega)}\leq Ch^2 \Big[\norm{u^0}_{H^2(\Omega)}+\int_0^{t^n}\norm{u_t}_{H^2(\Omega)}\ ds \Big] + C \ \Delta t_F \int_0^{t^n} \norm{u_{tt}}_{L^2(\Omega)} \ ds.
\label{normL2Parabol}
\end{equation}
\end{theorem}
On the energy error estimate, the following theorems hold.
\begin{theorem}[Corollary of Theorem 1.4 \cite{thomee2}]
\label{th14}
Let $\Omega$ be a convex polyhedron. Let $u \in H^{1}(0,T; H^1(\Omega))\cap L^2(0,T;H^2(\Omega))$ be the solution of \eqref{heatEq2} with $u^0 \in H^2(\Omega)$ and $u_h$ be the semidiscretized variational form \eqref{varpara1disc}. We have
\begin{equation}
\forall t \geq 0, \quad \norm{\nabla u(t)-\nabla u_h(t)}_{L^2(\Omega)} \leq  C(\mu)h \Big[\norm{u^0}_{H^2(\Omega)}+\norm{u(t)}_{H^2(\Omega)}+(\int_0^t \norm{u_t}_{H^1(\Omega)}^2 \ ds)^{1/2} \Big].
\label{Th2Parabol}
\end{equation}
\end{theorem}

The estimate \eqref{Th2Parabol} with the full discretization leads to the following theorem.  
\begin{theorem}
\label{th14bis}
Let $\Omega$ be a convex polyhedron. Let  $u \in H^{2}(0,T; H^1(\Omega)) \cap H^{1}(0,T; H^2(\Omega))$ be the solution of \eqref{heatEq2} with $u^0 \in H^2(\Omega)$, let $u_h^n$ be the fully-discretized solution of the variational form \eqref{varpara1disc}. We have
\begin{align}
 \forall n = 0,\dots N_T, \quad \norm{\nabla u_h^n-\nabla u(t^n)}_{L^2(\Omega)}\leq & C(\mu)h \Big[\norm{u^0}_{H^2(\Omega)}+(\int_0^{t^n}\norm{u_t}_{H^2(\Omega)}^2 \ ds)^{1/2} \Big]\nonumber \\
&+ C(\mu) \ \Delta t_F (\int_0^{t^n} \norm{\nabla u_{tt}}_{L^2(\Omega)}^2\ ds)^{1/2}.
\label{H1fullParabol}
\end{align}
\end{theorem}

\begin{proof}
  In \cite{thomee2}, these estimates are proven on the solution of the heat equation without a varying diffusion coefficient. These precised estimates can be obtained by following the same steps as in \cite{thomee2}. Let us detail e. g. the proof on the $H^1$ estimate of Theorem \ref{th14bis}.
We first decompose the error with two components $\theta$ and $\rho$ such that
\begin{align}
\forall n = 1,\dots N_T, \ e^n&:=\sqrt{\mu}(\nabla u_h^n-\nabla u(t^n))=\sqrt{\mu}((\nabla u^n_h-\nabla P^1_h u(t^n))+(\nabla P^1_h u(t^n)-\nabla u(t^n))), \nonumber \\
&=\sqrt{\mu}(\nabla \theta^n + \nabla \rho^n).
\label{errorParabol}
\end{align}

\begin{itemize}
\item For the estimate on $\rho^n$, a classical FEM estimate \cite{thomee2,FE} is
\begin{equation*}
\norm{ P_h^1v-v}_{L^2(\Omega)}+h\norm{\nabla (P_h^1v-v)}_{L^2(\Omega)}\leq Ch^2\norm{v}_{H^2(\Omega)},\quad \forall v\in H^2\cap H_0^1,
\end{equation*}
which leads to
\begin{equation*}
\norm{\nabla \rho^n}_{L^2(\Omega)}\leq Ch\norm{u(t^n)}_{H^2(\Omega)}, \ \forall n= 0,\dots N_T,
\end{equation*}
and it leads to , 
\begin{equation}
\norm{\nabla \rho^n}_{L^2(\Omega)}\leq Ch \Big[\norm{u^0}_{H^2(\Omega)}+\int_0^{t^n}\norm{u_t}_{H^2(\Omega)}\ ds \Big].
\label{rhoParabol}
\end{equation}

\item For the estimate on $\theta$, let us consider $ v \in V_h$. Since the operators $P_h^1$ and $\overline{\pt}$ commute, we may write
\begin{equation*}
(\overline{\pt}\theta^n,v)+\mu (\nabla \theta^n,\nabla v)=(\overline{\pt}u^n_h,v)-(P_h^1 \overline{\pt} u(t^n),v) +\mu (\nabla u_h^n,\nabla v)- \mu (\nabla P_h^1 u(t^n),\nabla v).
\end{equation*}
The weak formulations \eqref{varpara} and \eqref{varpara2disc} (fully-discretized solution with the Euler scheme) imply
\begin{align}
(\overline{\pt}\theta^n,v)+\mu (\nabla \theta^n,\nabla v)&=(f,v)-(P_h^1\overline{\pt} u(t^n),v) - \mu (\nabla P_h^1 u(t^n),\nabla v), \nonumber \\
&=(f,v)-(P_h^1\overline{\pt} u(t^n),v) - \mu (\nabla u(t^n), \nabla v), \textrm{ by definition of $P_h^1$}, \nonumber \\
&=(u_t(t^n),v)-(P_h^1\overline{\pt} u(t^n),v).\nonumber
\end{align}
Then, with the triangle inequality, it yields
\begin{align}
(\overline{\pt}\theta^n,v)+\mu (\nabla \theta^n,\nabla v)&=-((P_h^1-I)\overline{\pt}u(t^n),v)-((\overline{\pt}u(t^n)-u_t(t^n)),v) \nonumber \\
&:= -(w_1^n+ w_2^n,v)=-(w^n,v).
\label{teta2Parabol}
\end{align}

Instead of replacing $v$ by $\theta^n$ as in the $L^2$ estimate, here we replace $v$ by $ \overline{\pt} \theta^n,$ thus the equation \eqref{teta2Parabol} takes the form
\begin{equation*}
(\overline{\pt}\theta^n,\overline{\pt}\theta^n)+\mu (\nabla \theta^n,\overline{\pt} \nabla \theta^n) = - (w^n, \overline{\pt} \theta ^n).
\end{equation*}
Therefore, by definition of $\overline{\pt}$ for the Backward Euler discretization,
\begin{equation*}
\underbrace{(\overline{\pt}\theta^n,\overline{\pt}\theta^n)+\mu \frac{\norm{\nabla \theta^n}_{L^2(\Omega)}^2}{\Delta t_F} -\mu \frac{(\nabla \theta^n, \nabla \theta^{n-1})}{\Delta t_F} }_{T_a}= - (w^n, \overline{\pt} \theta ^n).
\end{equation*}
Young's inequality yields
\begin{equation*}
(\nabla \theta^n,\nabla \theta^{n-1})\leq \frac{1}{2}\norm{\nabla \theta^n}_{L^2(\Omega)}^2+  \frac{1}{2}\norm{\nabla \theta^{n-1}}_{L^2(\Omega)}^2,
\end{equation*}
therefore
\begin{equation*}
\norm{\overline{\pt}\theta^n}_{L^2(\Omega)}^2 +\mu \frac{\norm{\nabla \theta^n}_{L^2(\Omega)}^2}{2\Delta t_F} -\mu\frac{\norm{\nabla \theta^{n-1}}^2_{L^2(\Omega)}}{2 \Delta t_F} \ \leq \  T_a  \ \leq \  \frac{1}{2} \norm{w^n}^2_{L^2(\Omega)} + \frac{1}{2}\norm{\overline{\pt} \theta ^n}^2_{L^2(\Omega)},
\end{equation*}
and it results in 
\begin{equation*}
\norm{\overline{\pt}\theta^n}^2_{L^2(\Omega)} +\mu \frac{\norm{\nabla \theta^n}^2_{L^2(\Omega)}}{\Delta t_F}  \leq \mu \frac{\norm{\nabla \theta^{n-1}}^2_{L^2(\Omega)}}{ \Delta t_F} + \norm{w^n}^2_{L^2(\Omega)}.
\end{equation*}
Since $\norm{\overline{\pt}\theta^n}^2_{L^2(\Omega)} \geq 0$, it follows that 
\begin{equation*}
\forall n= 1,\dots,N_T, \ \norm{\nabla \theta^n}_{L^2(\Omega)}^2\leq \norm{\nabla \theta^{n-1}}_{L^2(\Omega)}^2 +\frac{\Delta t_F}{\mu} \norm{ w^n}_{L^2(\Omega)}^2,
\end{equation*}
and we recursively obtain
\begin{equation*}
\forall n= 1,\dots, N_T, \  \norm{\nabla \theta^n}_{L^2(\Omega)}^2\leq \norm{\nabla \theta^0}_{L^2(\Omega)}^2 + \frac{\Delta t_F}{\mu} \sum_{j=1}^n \norm{ w^j}_{L^2(\Omega)}^2.
\end{equation*}
By definition of $\theta$ (and $P_h^1$),
\begin{align*}
\norm{\nabla \theta^0}_{L^2(\Omega)}&=\norm{\nabla u^0_h-\nabla P_h^1 (u^0)}_{L^2(\Omega)} \leq \norm{\nabla u_h^0- \nabla u(t^0)}_{L^2(\Omega)}+\norm{\nabla u^0-\nabla P_h^1 (u^0)}_{L^2(\Omega)} \nonumber \\
&\leq   \norm{\nabla u_h^0- \nabla u^0}_{L^2(\Omega)}+ Ch\norm{u^0}_{H^2(\Omega)}.
\end{align*}

It remains to estimate the $L^2$ norm of the $w^j$, defined by \eqref{teta2Parabol}.

\begin{itemize}
\item Let us first consider the construction for $w_1$
  \begin{align}
w_1^j&=(P_h^1-I)\overline{\pt}u(t^j) \nonumber \\
&=\frac{1}{\Delta t_F}(P_h^1-I)\int_{t^{j-1}}^{t^j}u_t \ ds, \nonumber \\
&=\frac{1}{\Delta t_F}\int_{t^{j-1}}^{t^j} (P_h^1-I) u_t \ ds, \textrm{ since $P_h^1$ and the time integral commute}. \nonumber
\end{align}
  Thus, from Hölder's inequality,
  \begin{align}
  \frac{\Delta t_F}{\mu} \sum_{j=1}^n \norm{ w^j_1}_{L^2(\Omega)}^2 &\leq \frac{\Delta t_F}{\mu} \sum_{j=1}^n \int_{\Omega} \big[\frac{1}{\Delta t_F^2}\ \int_{t^{j-1}}^{t^j} ((P_h^1-I) u_t)^2 \ ds \ \Delta t_F \big] \nonumber\\
  &\leq \frac{1}{\mu} \sum_{j=1}^n \int_{t^{j-1}}^{t^j} \norm{(P_h^1-I) u_t}^2_{L^2(\Omega)} \ ds, \nonumber \\
&\leq \frac{C}{\mu}h^4 \sum_{j=1}^n \int_{t^{j-1}}^{t^j} \norm{u_t}_{H^2(\Omega)}^2, \textrm{ by the definition of ,}P_h^1 \nonumber \\
&\leq \frac{C}{\mu}h^4 \int_0^{t^n} \norm{u_t}_{H^2(\Omega)}^2 \ ds. 
\end{align}
\item To estimate the $L^2$ norm of the $w_2$, we write
\begin{align*}
w_2^j&=\frac{1}{\Delta t_F} (u(t^j)-u(t^{j-1}))-u_t(t^j), \nonumber\\
&=-\frac{1}{\Delta t_F}\int_{t^{j-1}}^{t^j} (s-t^{j-1})u_{tt}(s) \ ds, \nonumber
\end{align*}

such that we end up with 
\begin{equation*}
  \frac{\Delta t_F}{\mu} \sum_{j=1}^n \norm{ w_2^j}_{L^2(\Omega)}^2 \leq \frac{1}{\mu} \sum_{j=1}^n\norm{ \int_{t^{j-1}}^{t^j} (s-t^{j-1})  u_{tt}(s) \ ds}^2_{L^2(\Omega)} \leq \frac{ \Delta t_F^2}{\mu} \int_0^{t^n} \norm{  u_{tt}}_{L^2(\Omega)}^2 \ ds.
\end{equation*}
Combining the estimates on $\rho$ and $\theta$ concludes the proof.
\end{itemize}
\end{itemize}
\end{proof}
Finally, using the Crank-Nicolson scheme, we can recover the estimate in $H^2$ and $\Delta t_G^2$ in the $L^2$ norm

\begin{theorem}[Corollary of Theorem 1.6 \cite{thomee2}]
  \label{thCrankNicolson}
  Let $\Omega$ be a convex polyhedron. Let $u \in H^{2}(0,T; H^2(\Omega))\cap H^3(0,T;L^2(\Omega))$ be the solution of \eqref{heatEq2} with $u^0 \in H^2(\Omega)$. Let $u^m_H$ be the solution given by \eqref{varpara2discCN}, associated to Crank-Nicolson discretization on the time and spatial coarse grids. Let $\norm{u^0_H - u^0}_{L^2(\Omega)} \leq C H^2 \norm{u^0}_{H^2(\Omega)}$, then
\begin{align}
  &\forall m = 0,\dots,M_T, \quad \norm{u(\widetilde{t}^m)-u^m_H}_{L^2(\Omega)}\leq C \ H^2 \Big[\norm{u^0}_{H^2(\Omega)}+\int_0^{\widetilde{t}^m}\norm{u_t}_{H^2(\Omega)}\ ds\Big] \notag \\
  & \quad \quad \quad + C\ \Delta t_G^2 \big[(\int_0^{\widetilde{t}^m} \norm{u_{ttt} \ ds}_{L^2(\Omega)}^2)^{1/2} + (\int_0^{\widetilde{t}^m} \norm{\Delta u_{tt}}_{L^2(\Omega)}^2 \ ds)^{1/2} \big].
\label{th16}
\end{align}
\end{theorem}

Now, let $\widetilde{u_H}^n$ be the quadratic interpolation in time of the coarse solution at time $t^n \in I_m=[\widetilde{t}^{m-1},\widetilde{t}^m]$ defined on $[\widetilde{t}^{m-2},\widetilde{t}^m]$ from the values $u_H^{m-2}, u_H^{m-1},$ and  $u_H^m$,\ for all $m=2,\dots, M_T$.
To this purpose, we define the parabola on $[\wt^{m-2},\wt^m]$ with the values $u_H^{m-2}, u_H^{m-1}, u_H^m$: \\
    For $m\geq 2$, $\forall n \in I_m=[\widetilde{t}^{m-1},\widetilde{t}^m]$,
    \begin{align}
      \label{parabola}
      \widetilde{u_H}^n(\mu)=\frac{u_H^{m-2}(\mu)}{(\widetilde{t}^m - \wt^{m-2})(\wt^{m-2}-\wt^{m-1})}\big[ -(t^n)^2 + (\wt^{m-1} + \wt^m)t^n - t^{m-1}t^m \big] \nonumber \\
      + \frac{u_H^{m-1}(\mu)}{(\widetilde{t}^{m-2} - \wt^{m-1})(\wt^{m-1}-\wt^{m})}\big[ -(t^n)^2 + (\wt^{m} + \wt^{m-2})t^n - t^{m}t^{m-2} \big] \nonumber \\
    +  \frac{u_H^{m}(\mu)}{(\widetilde{t}^{m-1} - \wt^{m})(\wt^{m}-\wt^{m-2})}\big[ -(t^n)^2 + (\wt^{m-2} + \wt^{m-1})t^n - t^{m-2}t^{m-1} \big].
     \end{align}
    For $t^n \in I_1=[\wt^0,\wt^1]$, we use the same parabola defined by the values $u_H^0, u_H^{1}, u_H^2$ as the one used over $[\wt^1,\wt^2]$. Note that we choose this interpolation in order to keep an approximation of order 2 in time $\Delta t_G$ (it works also with other quadratic interpolations).
With this interpolated approximation, we have the following result.
\begin{corollary}[of Theorem \ref{thCrankNicolson}]
  Under the assumptions of Theorem \ref{thCrankNicolson}, let $\widetilde{u_H}^n$ be the quadratic interpolation in time of the coarse solution, defined above, then
\begin{align}
  &\forall n = 0,\dots, N_T, \quad \norm{u(\widetilde{t}^n)-\widetilde{u_H}^n}_{L^2(\Omega)}\leq C(\mu)H^2 \Big[\norm{u^0}_{H^2(\Omega)}+\int_0^{\widetilde{t}^m}\norm{u_t}_{H^2(\Omega)}\ ds\Big] \notag \\
  & \quad \quad \quad + C(\mu)\Delta t_G^2 [(\int_0^{\widetilde{t}^m} \norm{u_{ttt} \ ds}_{L^2(\Omega)}^2)^{1/2} + (\int_0^{\widetilde{t}^m} \norm{\Delta u_{tt}}_{L^2(\Omega)}^2 \ ds)^{1/2}].
 \label{coro}
\end{align}
\end{corollary}
Let us proceed with the NIRB algorithm description in the context of parabolic equations. 
\section{The Non-Intrusive Reduced Basis method (NIRB) in the context of parabolic equations.}
\label{NIRB}
\subsection{Main steps.}

This section describes the main steps of the two-grids method algorithm in the context of parabolic equations, and especially, how to define the RB using a POD-Greedy algorithm \cite{haasdonk2008reduced,haasdonk2013convergence,knezevic2010certified}.
Indeed, for evolution PDE's, a single solution associated with a parameter $\mu \in \mathcal{G}$ is made up of a sequence of potentially hundreds of snapshots over time (each snapshot being an HF finite element approximation in space at a time $t^n,\ n=0,\dots,N_T$).
As a result, each greedy step in the greedy algorithm is combined with a temporal compression step performed by a POD.
Let us go over the different steps of our offline-online decomposition.
The first three points are completed offline, while the remaining points are executed online.
\begin{itemize}
  \item ``Offline step''\\
\begin{enumerate}
\item From the training parameters $(\mu_i)_{i \in \{1,\dots , train\}}$, we compute fine snapshots $\{\uu_{h}^n(\mu_i)\}_{i \in \{1,\dots N\}}$ with the HF solver (solving problem \eqref{varpara2disc}). We define $\mathcal{G}_{train}= \underset{i \in \{1,\dots, train\}}{\cup} \mu_i$. 
\item We generate the RB functions (time-independent) $(\Phi^{h}_{i})_{i=1,\dots,N}$ through a POD-Greedy algorithm from the above snapshots, as presented in algorithm \ref{algoPODGreedy} below (or a full Greedy algorithm \ref{algogreedy}). \\

\begin{algorithm}[H]
  \caption{POD-Greedy algorithm } \label{algoPODGreedy}
  \hspace*{\algorithmicindent} \textbf{Input:} $N_{max}$, \ $\{\uu_h^n(\mu_1),\cdots,\uu_h^n(\mu_{N_{train}}) \textrm{ with } \mu_i  \in \mathcal{G}_{train}, \ n=0,\dots, N_T \}.$ \\
  \hspace*{\algorithmicindent} \textbf{Output:} Reduced basis $\{\Phi_1^h,\cdots,\Phi_{N}^h\}, N \leq N_{max}.$ \\
  \begin{algorithmic}
    \State Choose $\mu_1=\underset{ \mu \in \mathcal{G}_{train}}{\textrm{arg} \ \max}\  \norm{\uu_h^n(\mu)}_{l^{\infty}(0,\dots,N_T;\ L^2(\Omega))}$ ,
    \State Then produce the modes $\{ \Phi_1^h, \cdots, \Phi_{N_1}^h\} $ through a POD on $\{\uu_h^n(\mu_1), \ n=0,\dots,N_T \}$.
     \State Set $ \mathcal{G}_1={\mu_1}$ and $X_h^1= \spann \{ \Phi_1^h, \cdots, \Phi_{N_1}^h \}.$ 
 \While{$\overset{N}{\underset{k=2}{\sum}} N_k<N_{max}$} 
 \State  Choose $\mu_k=$ arg $\underset{\mu \in \mathcal{G}_{train}\backslash \mathcal{G}_{k-1}}{ \textrm{max}} \frac{\norm{ \uu_h^n ( \mu ) - P^{k-1} ( \uu_h^n(\mu) ) }_{l^{\infty}(0,\dots,N_T;L^2(\Omega))}}{\norm{\uu_h^n(\mu)}_{l^{\infty}(0,\dots,N_T;L^2(\Omega))}}$, with $P^{k-1}(\uu_h^n(\mu)):=\overset{N_{k-1}}{\underset{i=1}{\sum}} (\uu_h^n(\mu),\Phi_i^h)_{L^2} \Phi_i^h.$
 \State Then produce the modes $\{\Phi^h_{N_{k-1}+1}, \dots, \Phi^h_{N_k}\}$ through a POD on $\{\uu_h^n(\mu_k) - P^{k-1} ( \uu_h^n(\mu) ), \ n=0,\dots, N_T\}$.
   \State Set $ \mathcal{G}_k=\mathcal{G}_{k-1}\cup{\mu_k}$ and $X_h^k=X_h^{k-1} \oplus \spann \{\Phi^h_{k-1}, \cdots, \Phi^h_{N_k} \}.$ 
  \EndWhile
  \end{algorithmic}
\end{algorithm}

\begin{remark}
  In the following standard greedy algorithm, a tolerance treshold is used instead of a priori given number of basis functions. \\
  \begin{algorithm}[H]
      \caption{Greedy algorithm }  \label{algogreedy}
      \hspace*{\algorithmicindent} \textbf{Input:} $tol$, \ $\{\uu_h^n(\mu_1),\cdots,\uu_h^n(\mu_{N_{train}}) \textrm{ with } \mu_i  \in \mathcal{G}_{train}, \ n=0,\dots, N_T \}.$ \\
      \hspace*{\algorithmicindent} \textbf{Output:} Reduced basis $\{\Phi_1^h,\cdots,\Phi_{N}^h\}$ \\
    \begin{algorithmic}
      \State Choose $\mu_1,n_1=\underset{ \mu \in \mathcal{G}_{train},\  n \in \{ 0,\dots,N_T \}}{\textrm{arg} \max} \norm{\uu_h^n(\mu)}_{L^2(\Omega)}$ ,
      \State   Set $\Phi_1^h=\frac{\uu_h^{n_1}(\mu_1)}{\norm{\uu_h^{n_1}(\mu_1)}_{L^2}}$ \; 
\State  Set $ \mathcal{G}_1=\{\mu_1, n_1 \}$ and $X_h^1=\spann\{ \Phi_1^h \}$.
  \For{$k=2$ to $N$}:
  \State  $\mu_k,n_k=$ arg $ \underset{(\mu,\ n) \in (\mathcal{G}_{train}\times \{0,\dots,N_T\}) \backslash \mathcal{G}_{k-1}}{ \textrm{max}} \norm{ \uu_h^n ( \mu ) - P^{k-1} ( \uu_h^n(\mu) ) }_{L^2}$,  with $P^{k-1}$ defined as in POD-Greedy algorithm.
  \State Compute $\widetilde{\Phi_k^h}=\uu_h^{n_k}(\mu_k)-\overset{k-1}{\underset{i=1}{\sum}}(\uu_h^n(\mu_k),\Phi_i^h)_{L^2(\Omega)} \Phi_i^h$ and set $\Phi_k^h=\frac{\widetilde{\Phi^h_k}}{\norm{\widetilde{\Phi^h_k}}_{L^2(\Omega)}}$ 
    \State Set $\mathcal{G}_k=\mathcal{G}_{k-1}\cup \{ \mu_k\}$ and $X_h^k=X_h^{k-1} \oplus \spann\{\Phi^h_k \} $ 
   \State    Stop when $\norm{\uu_h^n(\mu)-P^{k-1}(\uu_h^n(\mu))}_{L^2}\leq tol, \ \forall \mu \in \mathcal{G}_{train},\ \forall n= 0,\dots, N_T. $ 
  \EndFor
  \end{algorithmic}
\end{algorithm}
Note that the greedy algorithm is generally less expensive (thanks to a-posteriori error estimates for stationnary problems). Yet, for time dependent problems, the POD-greedy is more reasonable when the snapshots are computed for all time steps. 
\end{remark}
 \begin{remark}
 The term 
 \begin{equation}
 \label{greedyestimateeq}
 \norm{ \uu_h^n( \mu ) - P^{k-1} ( \uu_h^n(\mu) ) }_{L^2(\Omega)}
 \end{equation}
 can be calculated either with a set of training snapshots as presented in \ref{algogreedy} or evaluated with an a-posteriori estimate. Since at each step $k$, all sets added in the basis are in the orthogonal complement of $X_h^{k-1}$, it yields an $L^2$ orthogonal basis without further processing.
 In practice, the algorithm is halted with a stopping criterion such as an error threshold or a maximum number of basis functions to generate.
 \end{remark}
 
 Then, we solve the following eigenvalue problem:
    \begin{numcases}
      \strut \textrm{Find } \Phi^{h} \in X_h^{N}, \textrm{ and } \lambda \in \mathbb{R} \textrm{ such that: }    \nonumber\\
        \forall \vv \in X_h^{N}, \int_{\Omega} \nabla \Phi^{h} \cdot \nabla \vv \ d\xx= \lambda \int_{\Omega} \Phi^{h} \cdot \vv \ d\xx, \label{orthohuhu} 
    \end{numcases} 
    where $X_h^{N} = Span\{ \Phi^h_1, \dots ,  \Phi^h_N\}$.
   We get an increasing sequence of eigenvalues $\lambda_i$, and orthogonal eigenfunctions $(\Phi^{h}_{i})_{i=1,\cdots,N}$, which do not depend on time, orthonormalized in $L^2(\Omega)$ and orthogonalized in $H^1(\Omega)$. Note that with Gram-Schmidt procedure, we only obtain an $L^2$-orthonormalized RB. 
  \item For the rectification post-treatment, we generate the equivalent coarse snapshots and the rectification matrix with algorithm \ref{rectiffz12}. The coarse snapshots, which have the same parameters as for the HF snapshots, are quadratically interpolated in time \eqref{parabola}. We ressort to the following algorithm.
 \begin{algorithm}[H]
   \caption{Offline rectification post-treatment}  \label{rectiffz12}
   \hspace*{\algorithmicindent} \textbf{Input:}$\{\uu_h^n(\mu_1) \cdots\uu_h^n(\mu_{N_{train}}), \textrm{ with } \mu_i  \in \mathcal{G}_{train}, \ n=0,\dots, N_T \}$ and with the same parameter $\{\uu_H^m(\mu_1),\cdots,\uu_H^m(\mu_{N_{train}}), \textrm{ with } \mu_i  \in \mathcal{G}_{train} \subset \mathcal{G}, \ m=0,\dots, M_T \}$ \\
  RB $\{\Phi^h_i\}_{i=1,\dots,N}.$ \\
  \hspace*{\algorithmicindent} \textbf{Output:}Rectification matrix $R_{i,j}^n, 1\leq i,j\leq N, \ n=0,\dots, N_T. $ \\
   \begin{algorithmic}
  \State Realize the quadratic interpolation of the coarse snapshots in time, denoted $\widetilde{\uu_H}^n,\ \ n=0,\dots, N_T$ with \eqref{parabola}. 
  \For{$ n=0,\dots, N_T$}
  \State Calculate the fine and coarse coefficients
  \State $\forall i=1,\cdots,N, \textrm{ and }  \forall \mu_k \in  \mathcal{G}_{train},\  A^n_{k,i}=\int_{\Omega} \widetilde{\uu_H}^n(\mu_k) \cdot \Phi_i^h\ d\xx,\quad \textrm{and }  B^n_{k,i}=\int_{\Omega} \uu_h^n(\mu_k) \cdot \Phi_i^h\ d\xx$, 
\State For $ i=1, \cdots,N,$ set $\mathbf{R}_i^n=((\mathbf{A}^n)^T\mathbf{A}^n+\delta \mathbf{I}_{N})^{-1}(\mathbf{A}^n)^T \mathbf{B}^n_i.$ 
  \EndFor
  \end{algorithmic}
 \end{algorithm}
 \begin{remark}
   Every time step has its own rectification matrix. Indeed, in our experiments, the results obtained with a global rectification matrix were less accurate. Because we have several time steps for each parameter in $\mathcal{G}_{train}$, $Ntrain \leq N$ in our context. Hence, $ \forall n\in \{1,\dots,N_T\},\ \mathbf{A}^n \in \mathbb{R}^{Ntrain \times N}$ is a rectangular ``flat'' matrix, and $(\mathbf{A}^n)^T\mathbf{A}^n$ is not invertible and requires the parameter $\delta$ for the inversion. In previous studies, the parameter $\delta$ was used only as a regularization parameter. \\
  We also remark that with the rectification post-treatment, the standard greedy algorithm \ref{algogreedy} may leads to more accurate approximations. It comes from the fact that the coefficients of the matrix are directly derived from the snapshots in that case. 
   \end{remark}
 \end{enumerate}
\item ``Online step''\\
  \begin{enumerate}
    \setcounter{enumi}{3}
  \item Now for the online part, we solve the problem \eqref{heatEq2} on the coarse mesh $\mathcal{T}_H$ for a new parameter $\mu \in \mathcal{G}$ at each time step $m=0,\dots, M_T$. 
  \item   We quadratically interpolate in time the coarse solution on the fine time grid with \eqref{parabola}.
    \item  Then, we linearly interpolate $\widetilde{u_H}^n(\mu)$ on the fine mesh in order to compute the $L^2$-inner product with the basis functions. The approximation used in the two-grid method is
      \begin{equation}
        \label{NIRBapproximation0}
     \textrm{For }n=0,\dots, N_T, \quad   u_{Hh}^{N,n}(\mu):= \overset{N}{\underset{i=1}{\sum}}(\widetilde{u_H}^n(\mu),\Phi^{h}_{i})\ \Phi^{h}_{i},
      \end{equation}
      and with the rectification post-treatment step \cite{chakirrect,VF}, it becomes
         \begin{equation}
        \label{NIRBapproximation0rect}
       R^n[u_{Hh}^{N}](\mu):= \overset{N}{\underset{i=1}{\sum}}R_{ij}^n\ (\widetilde{u_H}^n(\mu),\Phi^{h}_{j})\ \Phi^{h}_{i},
         \end{equation}
         where $R^n$ is the rectification matrix at time $t^n$ (see algorithm \ref{rectiffz12}).
\end{enumerate}
 \end{itemize}

 In the next section, we prove the optimal error in $L^{\infty}(0,T;H^1(\Omega))$.
\section{NIRB error estimate with parabolic equations}
\label{NIRBproof}

\begin{paragraph}{Main result}
Our main result is the following theorem.
\begin{theorem}{(NIRB error estimate for parabolic equations.)}
\label{EllipticEst}
Let us consider the problem \ref{heatEq2} with its exact solution $u(\xx, t; \mu)$, and the full discretized solution $u_h^n(\xx;\mu)$ to the problem \ref{varpara2disc}. 
Let $(\Phi^{h}_i)_{i=1,\dots,N}$ be the $L^2$-orthonormalized and $H^1$-orthogonalized RB generated with the POD-greedy algorithm \ref{algoPODGreedy} or \ref{algogreedy} from the fine solutions of \eqref{varpara2disc}. 

Let us consider the NIRB approximation, defined by \eqref{NIRBapproximation0}.
Then, the following estimate holds
\begin{equation}
\forall n = 0,\dots, N_T, \ \norm{u(t^n)(\mu)-u_{Hh}^{N,n}(\mu)}_{H^1(\Omega)}\leq \varepsilon(N) + C_1(\mu) h +C_2(N) H^2 +  C_3(\mu) \Delta t_F +C_4(N) \Delta t_G^2,  \label{estest2}
\end{equation}
where $C_1, C_2, C_3$ and $C_4$ are constants independent of $h$ and $H$, $\Delta t_F$ and $\Delta t_G$. The term $\varepsilon$ depends on the Kolmogorov $N$-width and measures the error given by \eqref{truerror}. 
  \end{theorem}
If $H$ is such as $H^2 \sim h$, $\Delta t_G^2\sim \Delta t_F$, and $\varepsilon(N)$ is small enough, with $C_2(N)$ and $C_4(N)$ not too large, it results in an error estimate in $\mathcal{O}(h + \Delta t_F)$.
Theorem \ref{EllipticEst} then states that we recover optimal error estimates in $L^{\infty}(0,T;H^1(\Omega))$. 
\begin{remark}
 This theorem can be generalized to $\mathbb{P}_k$ FEM space, with $k>1$.
  \end{remark}

With the $L^2$ norm, we obtain the following theorem.
\begin{theorem}
\label{NIRBcontribFEML2}
With the same assumptions as in the theorem \ref{EllipticEst}, with the $L^2$ orthonormalized RB, the following estimate holds
\begin{equation}
  \forall n = 0,\dots, \frac{\Delta t_F}{T}, \ \norm{u(t^n)(\mu)-u_{Hh}^{N,n}(\mu)}_{L^2(\Omega)}\leq \varepsilon'(N) + C_1' (H^2+ \Delta t_G^2) + C_2' (h^2 + \Delta t_F),
    \label{estimationNIRBFEM1L2}
\end{equation}
where $C_1'$ and $C_2'$ are constants independent of $h$, $H$ and $N$, and $\varepsilon'$ depends on the Kolmogorov N-width, and corresponds to the $L^2$ error between the fine solution and its projection on the reduced space.
\end{theorem}
\end{paragraph}
\begin{remark}
  Note that now $C_2'$ does not depend on $N$, unlike $C_2$ or $C_4$ above.
  \end{remark}
We now go on with the proof of Theorem \ref{EllipticEst}.
\begin{proof}
  The NIRB approximation at time step $n=0,\dots,N_T$, for a new parameter $\mu \in \mathcal{G}$ is defined by \eqref{NIRBapproximation0}.
Thus, the triangle inequality gives
\begin{align}
\ \norm{u(t)(\mu)-u_{Hh}^{N,n}(\mu)}_{H^1(\Omega)}&\leq \norm{u(t)(\mu)- u_{h}^n(\mu)}_{H^1(\Omega)}+\norm{u_h^n(\mu)-u_{hh}^{N,n}(\mu)}_{H^1(\Omega)}+ \norm{u_{hh}^{N,n}(\mu)-u_{Hh}^{N,n}(\mu)}_{H^1(\Omega)} \nonumber \\
&=:T_1+T_2+T_3, 
\label{triangleinequalitytime}
\end{align}
where $u_{hh}^{N,n}(\mu)=\overset{N}{\underset{i=1}{\sum}}  (u_h^n(\mu),\Phi^h_i)\ \Phi^h_i$.\\
\begin{itemize}
    \item 
    The first term $T_1$ may be estimated using the inequality \eqref{H1fullParabol},
such that 
\begin{equation}
    \norm{u(t^n)(\mu)-u_h^n(\mu)}_{H^1(\Omega)}\leq C(\mu)\ (h + \Delta t_F).
    \label{cea2}
\end{equation}

\item We denote by $\mathcal{S}_h=\{ u_h^n(\mu,t),\mu~\in~\mathcal{G},\ n=0,\dots N_T \}$ the set of all the solutions. For our model problem, this manifold has a low complexity. It means that for an accuracy $\varepsilon=\varepsilon(N)$ related to the Kolmogorov $N$-width of the manifold $\mathcal{S}_h$, for any $\mu \in \mathcal{G}$, and any $n \in 0,\dots,N_T$, $T_2$ is bounded by $\varepsilon$ which depends on the Kolmogorov $N$-width.
\begin{equation}
  T_2=\norm{u_h^n(\mu)-\overset{N}{\underset{i=1}{\sum}}(u_h^n(\mu),\Phi^h_i)\ \Phi^h_i}_{H^1(\Omega)} \leq \varepsilon(N).\label{kolmoFEMnew}
  \end{equation}
\item  The third term $T_3$ depends on the method used to create the RB.

  \begin{enumerate}
    \item Let us first consider the greedy approach with a Gram-Schmidt procedure and an eigenvalue problem \eqref{orthohuhu}, which yields to an orthogonalization in $L^2$ and in $H_1$. Therefore,
  
    \begin{equation}
\label{term3ortho}
    \norm{u_{hh}^{N,n}-u_{Hh}^{N,n}}_{H^1(\Omega)}^2 = \overset{N}{\underset{i=1}{\sum}} |(u_h^n(\mu)-\widetilde{u_H}^n(\mu),\Phi^{h}_{i})|^2 \norm{\Phi^{h}_{i}}_{H^1(\Omega)}^2,
    \end{equation}
    where $\widetilde{u_H}^n(\mu)$ is the quadratic interpolation of the coarse snapshots on time $t^n$,\ $\forall n= 0,\dots, N_T$, defined by \eqref{parabola}.
From the RB orthonormalization in $L_2$, the equation \eqref{orthohuhu} yields
\begin{equation}
  \label{orthoplus}
    \norm{\Phi^{h}_{i}}_{H^1}^2:=\norm{\nabla \Phi^{h}_{i}}_{L^2(\Omega)}^2
    =\lambda_i \norm{\Phi^{h}_{i}}_{L^2(\Omega)}^2
    =\lambda_i \leq \underset{i=1,\cdots,N}{\max \ }\lambda_i=\lambda_N,
\end{equation}
such that the equation  \eqref{term3ortho} yields
\begin{equation}
  \label{ici}
     \norm{u_{hh}^{N,n}-u_{Hh}^{N,n}}_{H^1(\Omega)}^2    \leq    C \lambda_N\norm{u_h^n(\mu)-\widetilde{u_H}^n(\mu)}_{L^2(\Omega)}^2.
\end{equation}


Now by definition of $\widetilde{u_H}^n(\mu)$ and by corollary \ref{coro} and Theorem \ref{th15}, for $t^n \in I_m,$\\
\begin{equation}
  \label{Aubinlike}
  \norm{u_h^n(\mu)-\widetilde{u_H}^n(\mu)}_{L^2} \leq  C( H^2+ \Delta t_G^2+h + \Delta t_F ), 
\end{equation}
and  we end up for equation \eqref{ici} with
\begin{equation}
    \norm{u_{hh}^{N,n}-u_{Hh}^{N,n}}_{H^1(\Omega)} \leq C \sqrt{\lambda_N} (H^2 +  \Delta t_G^2 + h + \Delta t_F),
    \label{T3termParaboo}
\end{equation}
where $C$ does not depend on $N$. Combining these estimates \eqref{cea2}, \eqref{kolmoFEMnew} and \eqref{T3termParaboo}, we obtain the estimate \eqref{estest2}.
\item Now we consider only an $L^2$-orthonormalized basis, which we will denote $(\Psi_{h,i})_{i=1,\dots,N}$ (obtained by a Gram-Schmidt algorithm or with the Greedy-POD algorithm \ref{algoPODGreedy}).
  The functions $(\Psi_{h,i})_{i=1,\dots,N}$ and $(\Phi^{h}_{i})_{i=1,\dots,N}$ are both generators of $X_h^N$. Thus, there exists $(\gamma^{i})_{i=1,\dots,N}\in \mathbb{R}^N$ such that $\Psi_{h,i}=\overset{N}{\underset{j=1}{\sum}} \gamma_j^{i} \Phi^{h}_{i}$.
  By the $H^1$-orthogonality of the $(\Phi^{h}_{i})_{j=1,\dots,N}$, it follows
\begin{align}
    \norm{\Psi_{h,i}}^2_{H^1} &=  \overset{N}{\underset{j=1}{\sum}} |\gamma_j^{i}|^2 \norm{\Phi^{h}_{i}}^2_{H^1}, \nonumber \\
    &\leq  \lambda_N \overset{N}{\underset{j=1}{\sum}} |\gamma_j^{i}|^2 \norm{\Phi^{h}_{i}}_{L^2(\Omega)}^2 \textrm{ by equation \eqref{orthohuhu}},\nonumber \\
    &= \lambda_N \norm{\Psi_{h,i}}_{L^2(\Omega)}^2 \textrm{ by the $L^2$-orthogonality of the } (\Psi_{h,i}^n)_{i=1,\dots,N}. \label{basisorthovspasorthoplus}
\end{align}

From the estimate \eqref{basisorthovspasorthoplus} and the  $L_2$-orthonormalization of the RB, 
\begin{align}
\label{term3pluss}
    \norm{u_{hh}^{N,n}(\mu)-u_{Hh}^{N,n}(\mu)}_{H^1} 
    &\leq \overset{N}{\underset{i=1}{\sum}} |(u_h^n(\mu)-\widetilde{u_H}^n(\mu),\Psi_{h,i}^n)| \norm{\Psi_{h,i}}_{H^1}, \nonumber\\
    & \leq C \sqrt{\lambda_N}  \overset{N}{\underset{i=1}{\sum}} |(u_h^n(\mu)-\widetilde{u_H}^n(\mu),\Psi_{h,i})|.
   \end{align}
 From Cauchy-Schwarz inequality, inequality \eqref{term3pluss} leads to
 \begin{align*}
 \label{term3Ndep1}
  \norm{u_{hh}^{N,n}(\mu)-u_{Hh}^{N,n}(\mu)}_{H^1} & \leq C \sqrt{\lambda_N} \sqrt{N} \sqrt{ \overset{N}{\underset{i=1}{\sum}}|(u_h^n(\mu)-\widetilde{u_H}^n(\mu),\Psi_{h,i})|^2} , \nonumber \\
    &\leq C \sqrt{\lambda_N}\sqrt{N}\norm{u_h^n(\mu)-\widetilde{u_H}^n(\mu)}_{L^2(\Omega)},
\end{align*}
 and we end up with
\begin{equation}
    \norm{u_{hh}^{N,n}(\mu)-u_{Hh}^{N,n}(\mu)}_{H^1}\leq C\sqrt{N}\sqrt{\lambda_N}(H^2+ \Delta t_G^2 + h + \Delta t_F),
    \label{T3FEM}
\end{equation}
which leads to estimate \eqref{estest2} using the inequalities \eqref{cea2}, \eqref{kolmoFEMnew}, and \eqref{T3FEM}, and concludes the proof.

  \end{enumerate}
\end{itemize}
\end{proof}
\begin{remark}
    Note that, from the proof of Theorem \ref{EllipticEst}, the estimate of the method implemented with an only $L^2$ orthonormalized basis set has an additional $\sqrt{N}$ factor (where $N$ is the number of modes) compared to the one obtained from the $L^2$ and $H^1$ orthogonalized basis set. Thus, the NIRB approximation is stabilized with the $H^1$ orthogonality, compared with a RB only orthogonalized in $L^2$. This difference may be numerically observed on more complex numerical results which require more modes \cite{grosj2022}.
\end{remark}

\begin{paragraph}{$L^2$ estimate.}
We proceed with the proof of theorem \ref{NIRBcontribFEML2}.
\begin{proof}
In analogy with the $H^1$ estimate, we have
\begin{align}
\forall n=0, \dots, N_T,\  \norm{u(t^n)(\mu)-u_{Hh}^{N,n}(\mu)}_{L^2}&\leq \norm{u(t^n)(\mu)- u_{h}^n(\mu)}_{L^2}+\norm{u_h^n(\mu)-u_{hh}^{N,n}(\mu)}_{L^2}+ \norm{u_{hh}^{N,n}(\mu)-u_{Hh}^{N,n}(\mu)}_{L^2} \nonumber \\
&=:T_1+T_2+T_3.
\label{triangleinequality1L2}
\end{align}
\begin{itemize}
\item For the first term $T_1$, it follows theorem \ref{th15} that
  \begin{equation}
    \label{T1L2}
T_1 \leq C (h^2+\Delta t_F).
\end{equation}
\item As with the $H^1$ estimate, $T_2$ can be estimated with the Kolmogorov N-width, and thus, for an accuracy $\varepsilon'= \varepsilon'(N) \leq \varepsilon(N)$ (where $\varepsilon(N)$ bounds the $H^1$ error)
  \begin{equation}
    \label{T2L2}
T_2 \leq \varepsilon'.
\end{equation}

\item For the last term $T_3$, by $L^2$-orthonormality,
\begin{align}
\norm{u_{hh}^{N,n}(\mu)-u_{Hh}^{N,n}(\mu)}_{L^2(\Omega)}^2&= \overset{N}{\underset{i=1}{\sum}}|(u_h^n(\mu)-u_H^n(\mu),\Psi_{h,i}^n)|^2 \norm{\Psi^n_{h,i}}_{L^2(\Omega)}^2, \nonumber\\
&\leq C\norm{u_h^n(\mu)-u_H^n(\mu)}_{L^2(\Omega)}^2.
\end{align}
Note that, the dependence in $N$ is removed in the previous inequality. 
By theorem \ref{th16} and triangle inequality, it leads to
\begin{equation}
\norm{u_{hh}^{N,n}(\mu)-u_{Hh}^{N,n}(\mu)}_{L^2(\Omega)}^2 \leq C \ (H^2+ \Delta t_G^2 + h + \Delta t_F).
\label{term3L2}
\end{equation}
Combining \eqref{triangleinequality1L2}  with \eqref{T1L2}, \eqref{T2L2}, \eqref{term3L2} concludes the proof. 
\end{itemize}
\end{proof}
\end{paragraph}

\section{Numerical results.}
\label{results}
 In this section, we have applied the NIRB algorithm on several numerical tests. For each case, we compare the plain NIRB errors (without the rectification post-treatment) with the rectified NIRB errors given by algorithm \ref{rectiffz12}:
  \begin{itemize}
  \item first, on the heat equation \eqref{heatEq2} with $\Delta t_G\simeq H \simeq 2\ h \simeq 2 \ \Delta t_F$. Note that in some situations, because of the constants $C_2$ and $C_4$ in the estimate of theorem \ref{EllipticEst}, the best size for the coarse mesh may not be $\Delta t_F^{1/2}$. 
    \item then, on the heat equation with $\Delta t_G^2\simeq H \simeq \sqrt{h} \simeq \Delta t_F$. 
    \item finally, we also tested our problem on a more complex problem, which is the Brusselator equations.
      \end{itemize}
  We have implemented both schemes (Euler and RK2) using FreeFem++ (version 4.9) \cite{bamg2} to compute the fine and coarse snapshots, and the solutions have been stored in VTK format (with $u^0=0$).
Then we have applied the plain NIRB and the NIRB rectified algorithm with python, in order to highlight the non-intrusive side of this method (as in \cite{grosj2022}). After saving the NIRB approximations with Paraview module on Python, the errors have been computed with FreeFem++. 

\subsection{The heat equation with $\Delta t_G\simeq H \simeq 2\ h \simeq 2 \ \Delta t_F$.}
We have taken the parameter set $\mathcal{G}=[0.5,9.5]$. \\
Not that for $\mu=1$, we can calculate an analytical solution, which is given by
\begin{equation}
  \label{mu1}
u(t,\xx;1)=10 t x^2(1-x)^2y^2(1-y)^2,
\end{equation}
for a right-hand side function
\begin{equation}
        f(t,\xx)= 10 [x^2(x-1)^2y^2(y-1)^2 - 2t((6x^2-6x+1)(y^2(y-1)^2)+(6y^2-6y+1)(x^2(x-1)^2))],
     \end{equation}
where $\xx=(x,y)$.

We have retrieved several snapshots on $t=[1,2]$ (note that the coarse time grid must belong to the interval of the fine one), and tried our algorithms on several size of meshes, always with $\Delta t_F \simeq h$ and $\Delta t_G \simeq H$ (both schemes are stables).

\begin{itemize}
\item We have first taken 18 parameters in $\mathcal{G}$ for the RB construction such that $\mu_i=0.5i, \ i=1,\dots, 19,\ i\neq 2,$ and the true solution \eqref{mu1}, with $\mu=1$. In what follows 
\figref{H10errorsmu1} and \figref{L2errorsmu1}, we present the errors of the FEM solutions and compare them to the ones obtained with the NIRB algorithms ((with POD-Greedy) to observe the convergence rate. 

We recall that the rectification post-processing step is done for each time step.
Thus, the NIRB with rectification is given by

\begin{equation}
  R^n[u_{Hh}^{N}](\mu)=\overset{N}{\underset{i,j=1}{\sum}}\ R_{ij}^n\ \alpha_j^H(\mu,t^n) \ \Phi^h_i(\xx), \ n \geq 0,
  \label{NIRBrec}
            \end{equation}
where the rectification matrix $R$ may be seen as a familly of 2nd-order tensors indexed by $n$.

The relative errors have been computed in the maximum-norm in time. The $H^1_0$ NIRB error is defined as

\begin{equation}
\label{errorParabolik}
\frac{\norm{u(1) - u_{Hh}^{N}(1)}_{l^{\infty}(0,\dots,N_T;H^1_0(\Omega))}}{\norm{u(1)}_{l^{\infty}(0,\dots,N_T;H^1_0(\Omega))}},
\end{equation}

and with the rectification post-treatment we have
\begin{equation}
\label{errorParabolikrect}
\frac{\norm{u(1) - R[u_{Hh}^{N}(1)]}_{l^{\infty}(0,\dots,N_T;H^1_0(\Omega))}}{\norm{u(1)}_{l^{\infty}(0,\dots,N_T;H^1_0(\Omega))}},
\end{equation}

where $R[u_{Hh}^N]$ is defined by \eqref{NIRBrec}, and these relative errors are compared to the FEM ones defined as

\begin{equation}
\label{errorParabolikFEM}
\frac{\norm{u(1) - u_{h}(1)}_{l^{\infty}(0,\dots,N_T;H^1_0(\Omega))}}{\norm{u(1)}_{l^{\infty}(0,\dots N_T;H^1_0(\Omega))}} \textrm { and } \frac{\norm{u(1) - u_{H}(1)}_{l^{\infty}(0,\dots,N_T;H^1_0(\Omega))}}{\norm{u(1)}_{l^{\infty}(0,\dots, N_T;H^1_0(\Omega))}}.
\end{equation}

\begin{figure} 
            \centering
            \includegraphics[scale=0.4]{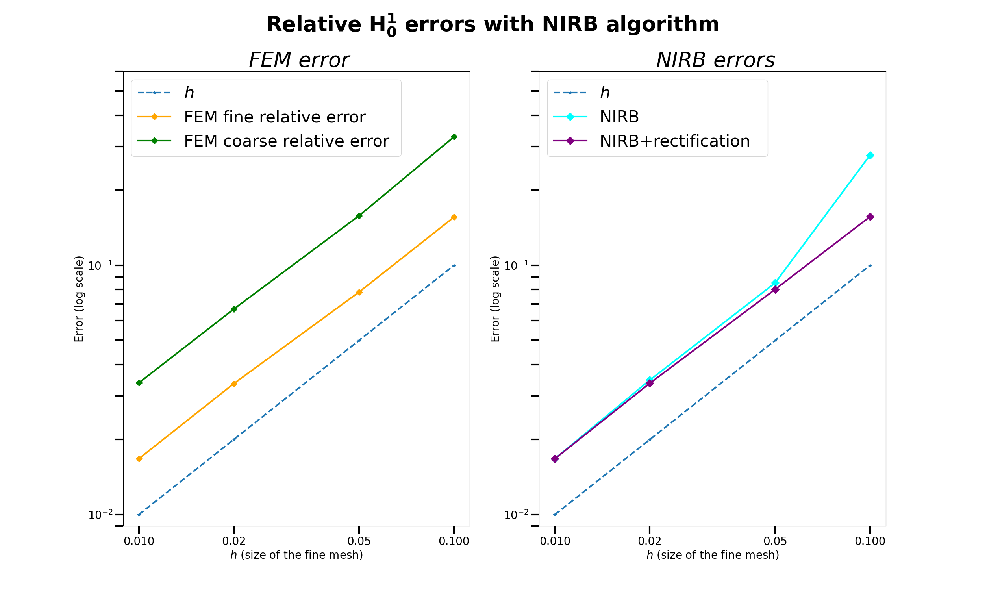}
            \caption{$\Delta t_G\simeq H \simeq 2\ h \simeq 2 \ \Delta t_F$. Convergence rate for $\mu=1$ (as a new parameter): FEM $H^1_0$ relative errors \eqref{errorParabolikFEM} for several sizes of mesh (left) compared to the NIRB method with ($N=3$) and without the rectification post-treatment ($N=3$) (right) \eqref{errorParabolikrect}}%
            \label{H10errorsmu1}
\end{figure}

We also plot the $L^2$ errors in \figref{L2errorsmu1}. We can see that the NIRB $L^2$ relative error without rectification is very close to the coarse $L^2$ relative error, thus, in the $L^2$ norm, no improvement is provided by the NIRB algorithm, however with the rectification post-treatment, the error reaches the fine accuracy.

\begin{figure}
            \centering
            \includegraphics[scale=0.4]{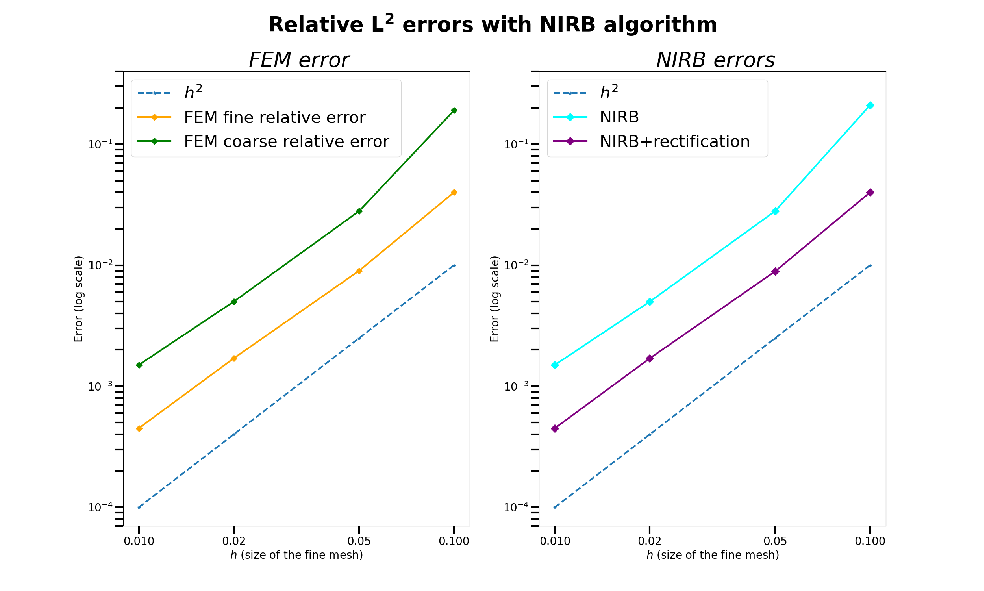}
            \caption{$\Delta t_G\simeq H \simeq 2\ h \simeq 2 \ \Delta t_F$. Convergence rate for $\mu=1$  (as a new parameter): FEM $L^2$ relative errors for several sizes of mesh (left) compared to NIRB with ($N=3$) and without the rectification post-treatment ($N=3$) (right) }%
            \label{L2errorsmu1}
\end{figure}

\item Then, we have taken 19 parameters in $\mathcal{G}$ for the RB construction such that $\mu_i=0.5i, \ i=1,\dots, 19$ and have applied the ``leave-one-out'' strategy. In order to evaluate the NIRB algorithm with respect to the parameters, table \ref{TableParam1} presents the maximum $H^1_0$-error of the NIRB rectified approximation over the parameters. The
error is given by
\begin{equation}
  \label{errfine}
  \underset{\mu \in \mathcal{G}_{train}}{\max} \frac{\norm{u_h(\mu) - R[u_{Hh}^{N}](\mu)}_{l^{\infty}(0,\dots N_T;H^1_0(\Omega))}}{\norm{u_h(\mu)}_{l^{\infty}(0, \dots , N_T;H^1_0(\Omega))}}.
\end{equation}
and the maximum in our training parameters is retrieved for $\mu=9$. 

 \begin{table}[tbhp]
     {\footnotesize
       \caption{ Maximum $H^1_0$ error over the parameters [$\mu=9$] \eqref{errfine} (compared to the true NIRB projection and to the FEM coarse projection) with $N=20$ with $h=0.01$}\label{TableParam1}
   \begin{center}
    \begin{tabular}{ |c| c| c|}
      \hline
 NIRB rectified error & $ \underset{\mu \in \mathcal{G}_{train}}{\max} \frac{\norm{u_h(\mu) - u_{hh}^{N}(\mu)}_{l^{\infty}(0,\dots,N_T;H^1_0(\Omega))}}{\norm{u_h(\mu)}_{l^{\infty}(0,\dots,N_T;H^1_0(\Omega))}}$ &   $\underset{\mu \in \mathcal{G}_{train}}{\max} \frac{\norm{u_h(\mu) - u_H(\mu)}_{l^{\infty}(0,\dots,N_T;H^1_0(\Omega))}}{\norm{u_h(\mu)}_{l^{\infty}(0,\dots,N_T;H^1_0(\Omega))}}$ \\
    \hline
  $4.84 \times 10^{-5} $  & $2.31\times 10^{-10}$ & $ 1.42 \times 10^{-1}$ \\
    \hline
    \end{tabular}
   \end{center}
   }
 \end{table}
\end{itemize}

  \subsubsection{ Time execution (min,sec)}
 \label{time}
 We present the FEM and NIRB runtimes in \ref{runtimes1} and \ref{runtimes2}.
   \begin{table}[tbhp]
     {\footnotesize
       \caption{ FEM runtimes}\label{runtimes1}
   \begin{center}
    \begin{tabular}{ |c| c| }
      \hline
   FEM high fidelity solver & FEM coarse solution \\
    \hline
     00:03 & 00:02\\
    \hline
    \end{tabular}
   \end{center}
   }
   \end{table}
   \begin{table}[tbhp]
     {\footnotesize
       \caption{NIRB runtimes ($N=10$)}\label{runtimes2}
   \begin{center}
    \begin{tabular}{|c|c|}
      \hline
      NIRB Offline & classical rectified NIRB online\\
      \hline
      1:45 & 00:02\\
      \hline
    \end{tabular}
   \end{center}
   }
   \end{table}

   \normalsize
   \subsection{The heat equation with $H^2 \simeq h \simeq \Delta t_G^2 \simeq \Delta t_F$}
  
   As previously, in \figref{H10errorsmu12} we display the convergence rate of the fine approximations (left) and of the NIRB approximations (with and without rectification). For all meshes, we choose $\mu=1$ and as expected, we observe that both NIRB errors converge in $\mathcal{O}(h+\Delta t_F)$, and we retrieved the fine accuracy with the rectified approximation.

\begin{figure}
            \centering
            \includegraphics[scale=0.4]{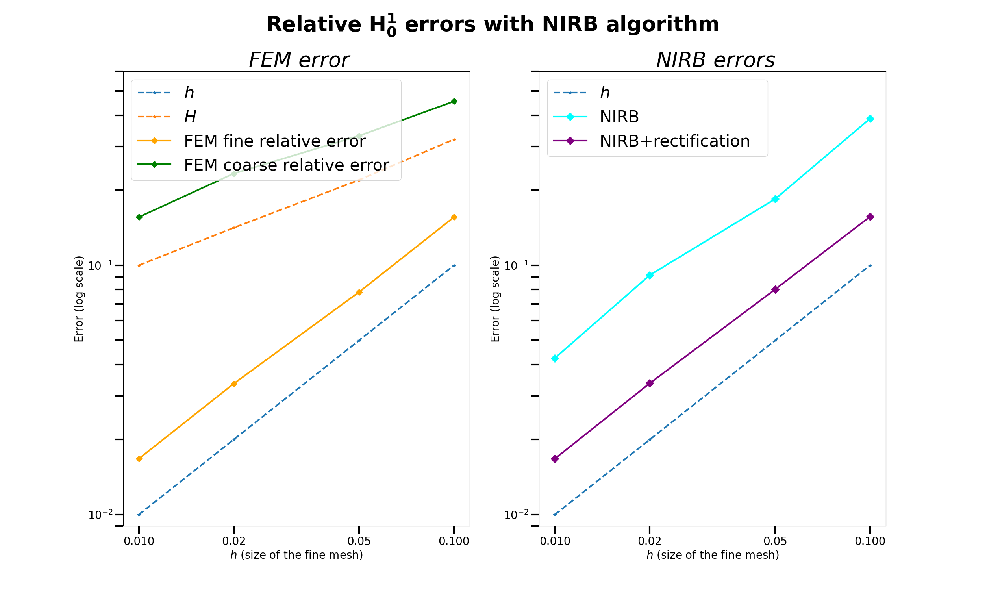}
            \caption{$H^2 \simeq h \simeq \Delta t_G^2 \simeq \Delta t_F$, Convergence rate for $\mu=1$ (as a new parameter): FEM relative $H^1_0$ errors for several sizes of mesh (left) compared to NIRB with ($N$=3) and without the rectification post-treatment ($N=3$) (right)}%
            \label{H10errorsmu12}
\end{figure}

We also plot the $L^2$ errors in \figref{L2errorsmu22}.

\begin{figure}
            \centering
            \includegraphics[scale=0.4]{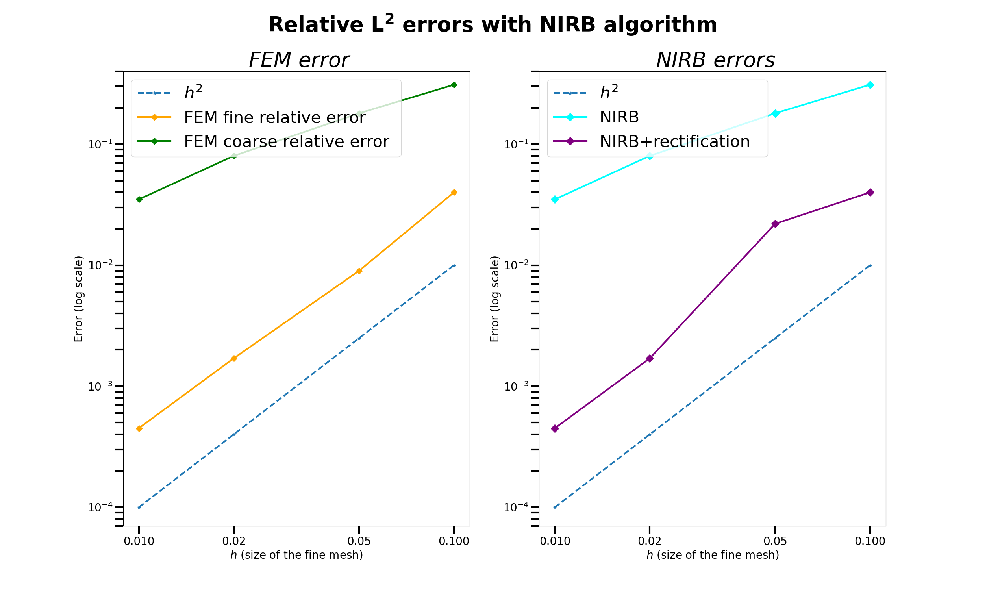}
            \caption{$H^2 \simeq h \simeq \Delta t_G^2 \simeq \Delta t_F$, Convergence rate for $\mu=1$ (as a new parameter): FEM $L^2$ relative errors for several sizes of mesh (left) compared to NIRB with ($N=3$) and without the rectification post-treatment ($N=3$) (right) }%
            \label{L2errorsmu22}
\end{figure}

Finally, in order to evaluate the NIRB algorithm with respect to the parameters, table \ref{TableParam2} presents the maximum $H^1_0$-error of the NIRB rectified approximation over the parameters. The
error is given by \eqref{errfine}, and the maximum in our training parameters is retrieved for $\mu=9$. 
 \begin{table}[tbhp]
     {\footnotesize
       \caption{ Maximum $H^1_0$ error over the parameters [$\mu=9$] \eqref{errfine} (compared to the true NIRB projection and to the FEM coarse projection) with $N=20$ with $h=0.01$}\label{TableParam2}
   \begin{center}
    \begin{tabular}{ |c| c| c|}
      \hline
 NIRB rectified error & $ \underset{\mu \in \mathcal{G}_{train}}{\max} \frac{\norm{u_h(\mu) - u_{hh}^{N}(\mu)}_{l^{\infty}(0,\dots,N_T;H^1_0(\Omega))}}{\norm{u_h(\mu)}_{l^{\infty}(0,\dots,N_T;H^1_0(\Omega))}}$ &   $\underset{\mu \in \mathcal{G}_{train}}{\max} \frac{\norm{u_h(\mu) - u_H(\mu)}_{l^{\infty}(0,\dots,N_T;H^1_0(\Omega))}}{\norm{u_h(\mu)}_{l^{\infty}(0,\dots,N_T;H^1_0(\Omega))}}$ \\
    \hline
  $4.21 \times 10^{-3} $  & $1.40\times 10^{-9}$ &  $7.80 \times 10 ^{-1}$ \\
    \hline
    \end{tabular}
   \end{center}
   }
 \end{table}
 
We observe that the errors without the rectification post-treatment increases with $N$ due to the role of the constants $C_2$ and $C_4$ in the estimate of theorem \ref{EllipticEst}, whereas with the post-treatment they remain stable. This is illustrated by \figref{NEVChanges} where the $H^1_0$ errors are displayed for $\mu=1$ and different number of modes $N$.

\begin{figure}
            \centering
            \includegraphics[scale=0.4]{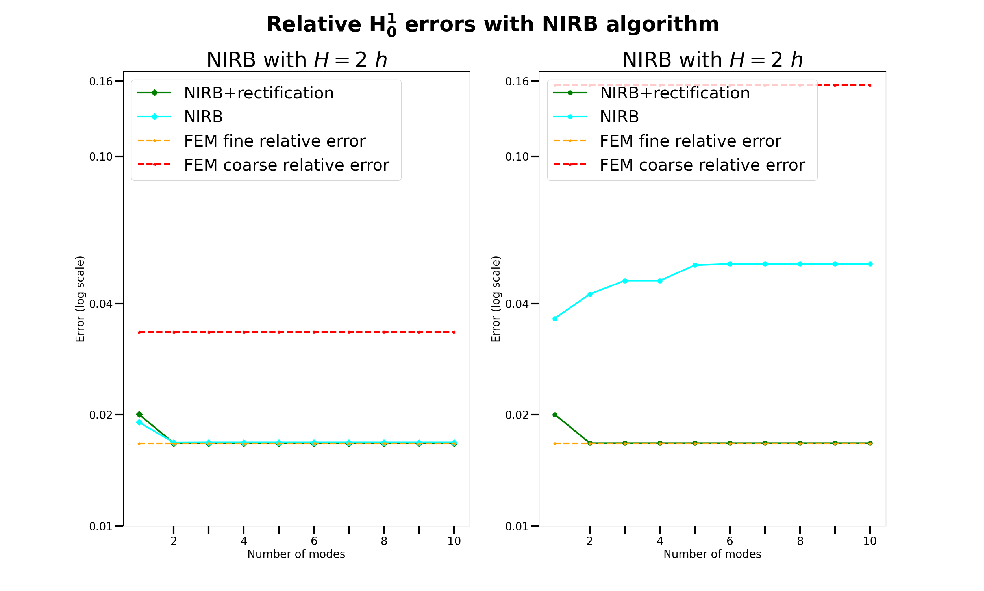}
            \caption{For $h=0.01$: $H=2h$ (left) , $H=h^2$ (right) $\mu=1$, NIRB relative $H^1_0$ errors and rectified NIRB (+ rectification post-treatment) $H^1_0$ compared to FEM errors with different modes $N$}%
            \label{NEVChanges}
\end{figure}
\newpage
\begin{remark}
We may also consider NIRB aproximations of \eqref{varpara2disc} under the form
\begin{equation}
  u_{Hh}^{N,n}(\xx; \mu)= \overset{N}{\underset{i=1}{\sum}} \alpha_i^H(\mu,t^n) \ \Phi_{h,i}^n(\xx), \ n \geq 0,
  \label{otherdecom}
\end{equation}
with $(\Phi_{h,i}^n)_{i=1,\dots,N}$ time-dependent basis functions. This time, the greedy algorithm \ref{algogreedy} is executed for each time step and thus, this method is less efficient (in term of storage) since we have to store $N$ times the number of time steps of the reduced basis.

With this decomposition, we obtained the following results (see \figref{NIRBPArabolbis}).
\begin{figure}
         \centering
         \includegraphics[scale=0.4]{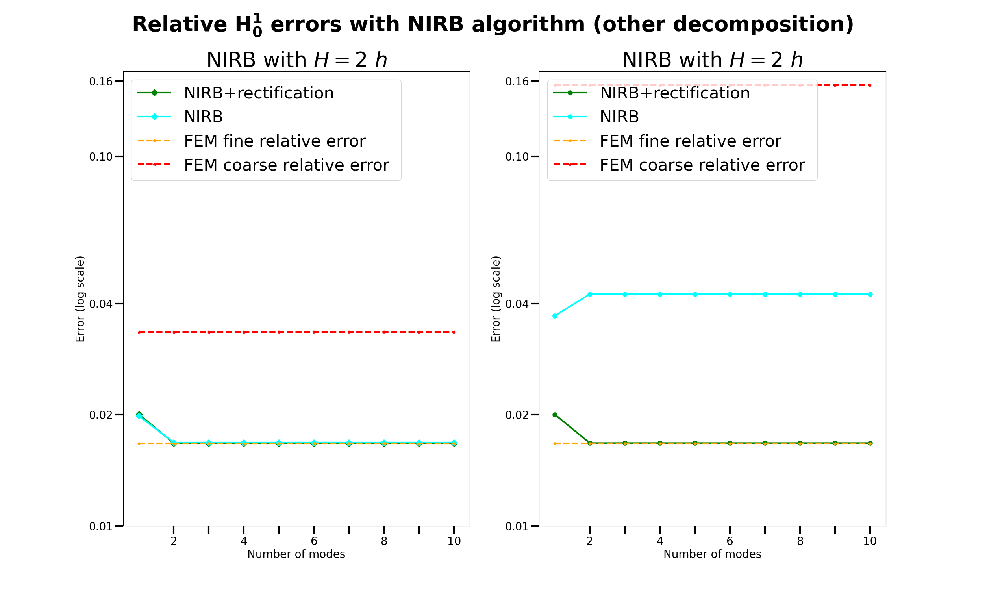}
           \caption{For $h=0.01$: $H=2h$ (left) , $H=h^2$ (right) $\mu=1$, NIRB relative $H^1_0$ errors and rectified NIRB (+ rectification  post-treatment) $H^1_0$ compared to FEM errors with different modes $N$ \textit{using the other NIRB decomposition \eqref{otherdecom}}}
             \label{NIRBPArabolbis}
\end{figure}

\end{remark}

  \subsubsection{ Time execution (min,sec)}
 \label{time2}
 We present the FEM and NIRB runtimes in \ref{runtimes12} and \ref{runtimes22}.
   \begin{table}[tbhp]
     {\footnotesize
       \caption{ FEM runtimes}\label{runtimes12}
   \begin{center}
    \begin{tabular}{ |c| c| }
      \hline
   FEM high fidelity solver & FEM coarse solution \\
    \hline
    00:03 & 00:01\\
    \hline
    \end{tabular}
   \end{center}
   }
   \end{table}
   \begin{table}[tbhp]
     {\footnotesize
       \caption{NIRB runtimes ($N=10$)}\label{runtimes22}
   \begin{center}
    \begin{tabular}{|c|c|}
      \hline
      NIRB Offline & classical rectified NIRB online\\
      \hline
      1:32 & 00:02\\
      \hline
    \end{tabular}
   \end{center}
   }
   \end{table}
   
\subsubsection{Comments on the results}
 \begin{itemize}
 \item On the first tests, with $\Delta t_G\simeq H \simeq 2\ h \simeq 2 \ \Delta t_F$, we observe that the NIRB
   \begin{itemize}
     \item  with and without the rectification post-treatment converge in $\mathcal{O}(h+ \Delta t_F)$, as expected from the estimates of Theorem \ref{EllipticEst} for the plain NIRB (see \figref{H10errorsmu1}).
     \item gives the same accuracy as with the HF solutions in the $H^1_0$ norm in both cases.
     \item yields an optimal $L^2$ error estimate for the NIRB with the rectification post-treatment, in $\mathcal{O}(h^2 + \Delta t_F)$, whereas the error for the plain NIRB is not enhanced by the NIRB algorithm (compared to the coarse FEM approximation), as predicted by Theorem \ref{NIRBcontribFEML2} (see \figref{L2errorsmu1}).
   \end{itemize}
   The rough mesh size is finer than $\sqrt{h}$. In that case, the plain NIRB algorithm is sufficient to retrieve the optimal $H^1$ accuracy.
 \item Then, on the heat equation with $\Delta t_G^2\simeq H \simeq \sqrt{h} \simeq \Delta t_F$, we remark that
   \begin{itemize}
   \item with both algorithms, the error converge in $\mathcal{O}(h+ \Delta t_F)$ (see \figref{H10errorsmu12}).
   \item The plain NIRB method allows us to reduce the $H^1_0$ error compared to the coarse FEM approximation. Yet, the fine accuracy is recovered only while adding the rectification post-treatment.
     \item We retrieved the rates of convergence expected from theorem \ref{NIRBcontribFEML2} in $\mathcal{O}(h^2+ \Delta t_F)$, yet only the NIRB with the rectification post-treatment yields the same errors as the HF ones (see \figref{L2errorsmu22}). 
 \end{itemize}
 \end{itemize}

\subsection{The parameterized Brusselator equations}
The Brusselator problem \cite{MITTAL20115404} involves chemical reactions. It is a more complex test from a simulation point of view. The chemical concentrations in this problem are controlled by parameters throughout the reaction process, making it an interesting application of a NIRB method.
Let us introduce the Brusselator problem in a spatial domain $\Omega=[0,1]^2$.
The nonlinear system of this two-dimensional reaction-diffusion problem writes
\begin{equation*}
\begin{cases}
   & \pt_t u_1  = a+ u_1u_2^2 -(b+1) u_1 + \alpha \Delta u_1,\textrm{ in }\Omega\times ]0,T] \\
    & \pt_t u_2  = bu_1 -u_1 \ u_2^2 +\alpha  \Delta u_2, \textrm{ in } \Omega \times ]0,T], \\
      &  u_1(\xx,0)=u^0(\xx)=2+0.25y,  \textrm{ in } \Omega \\
            & u_2(\xx,0)=v_0(\xx)=1 + 0.8x,  \textrm{ in } \Omega, \\
            & \pt_n u_1 =0 ,\ \pt \Omega,\\
  & \pt_n u_2 = 0,\  \pt \Omega.
\end{cases}
\end{equation*}
We now have to deal with a nonlinearity as well as two unknowns. Our parameter, denoted $\boldsymbol{\mu}=(a,b,\alpha)$, belongs to $[2,4]\times [1,4] \times [0.001,0.05]$. We have taken an ending time $T=5$. These parameters are standard. We note that, for $b\leq 1 + a^2$, the solutions are stable, and for $\alpha$ small enough, they converge to $(u_l,v_l)=(a,\frac{b}{a})$.\\
We use an Euler implicit scheme for fine solutions with the Newton algorithm to deal with nonlinearity and an explicit 2nd order Runge-Kutta scheme (RK2) for the coarse mesh.
Indeed, the solutions blow up with an explicit Euler scheme, whereas it remain stable for our parameters with an order 2 scheme (as RK2).\\

 Thus we took $a={2,2.5,4}$ , $b={1,3,4}$ and $\alpha={0.001, \ 0.005, \  0.01, \ 0.05}$, and tested our NIRB algorithm with the rectification post-treatment on the new parameter $(a,b,\alpha)=(3,2,0.008)$ with $h=0.02 = \Delta t_F \simeq H^2 = \Delta t_G^2$ ($\Delta T_G=H=0.1$).
   The relative $H_0^1$ errors of the NIRB approximation with rectification post-treatment \eqref{errorParabolikrect} and of the FEM fine approximation \eqref{errorParabolikFEM}) are displayed in \figref{fig: Bruss}. Here, we do not know the exact solutions but we observe a gain of accuracy of factor 20 with 30 modes on the relative $H_0^1$ error with the NIRB solutions compared to the coarse FEM one. 
   
   \begin{figure}
      \centering
\includegraphics[scale=0.4]{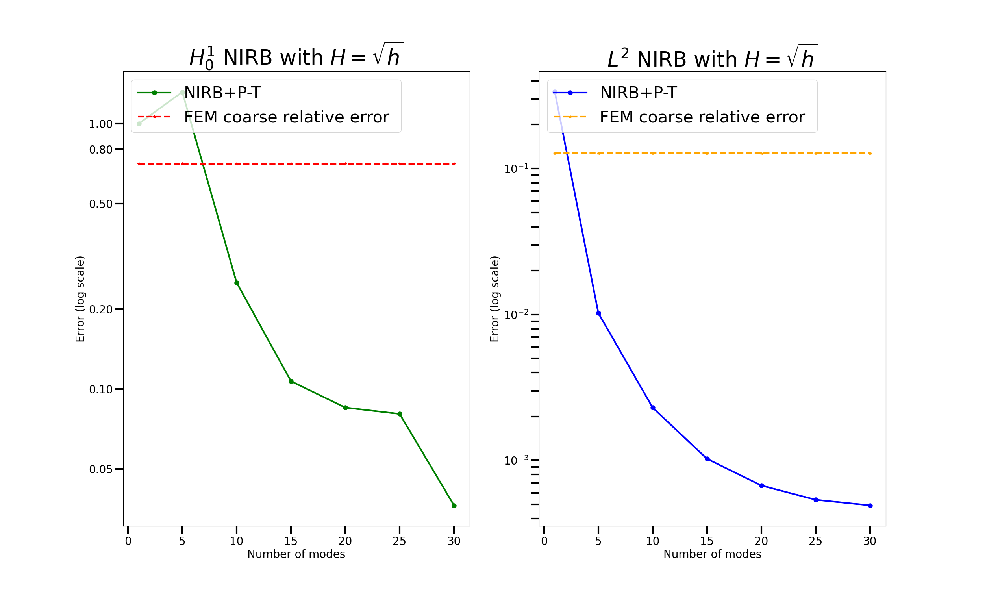}
\caption{Test with $l^{\infty}(0,\dots,N_T;H^1(\Omega))$ relative errors with a new parameter $(a,b,\alpha)=(3,2,0.008)$, $t_0=0$, $T=5$, $\Omega=[0,1] \times [0,1]$ (NIRB + rectification post-treatment compared to coarse FEM error)}
\label{fig: Bruss}
\end{figure}

   In \figref{fig:solNIRB} follows the NIRB rectified solution for $N=10$ modes at time $T=5$ for the two variables $u_1$ and $u_2$. The approximation is close to $(a,\frac{b}{a})=(3,2/3)$ as expected.\\

  \begin{figure}
     \centering
     \begin{subfigure}[b]{0.42\textwidth}
         \centering
         \includegraphics[scale=0.16]{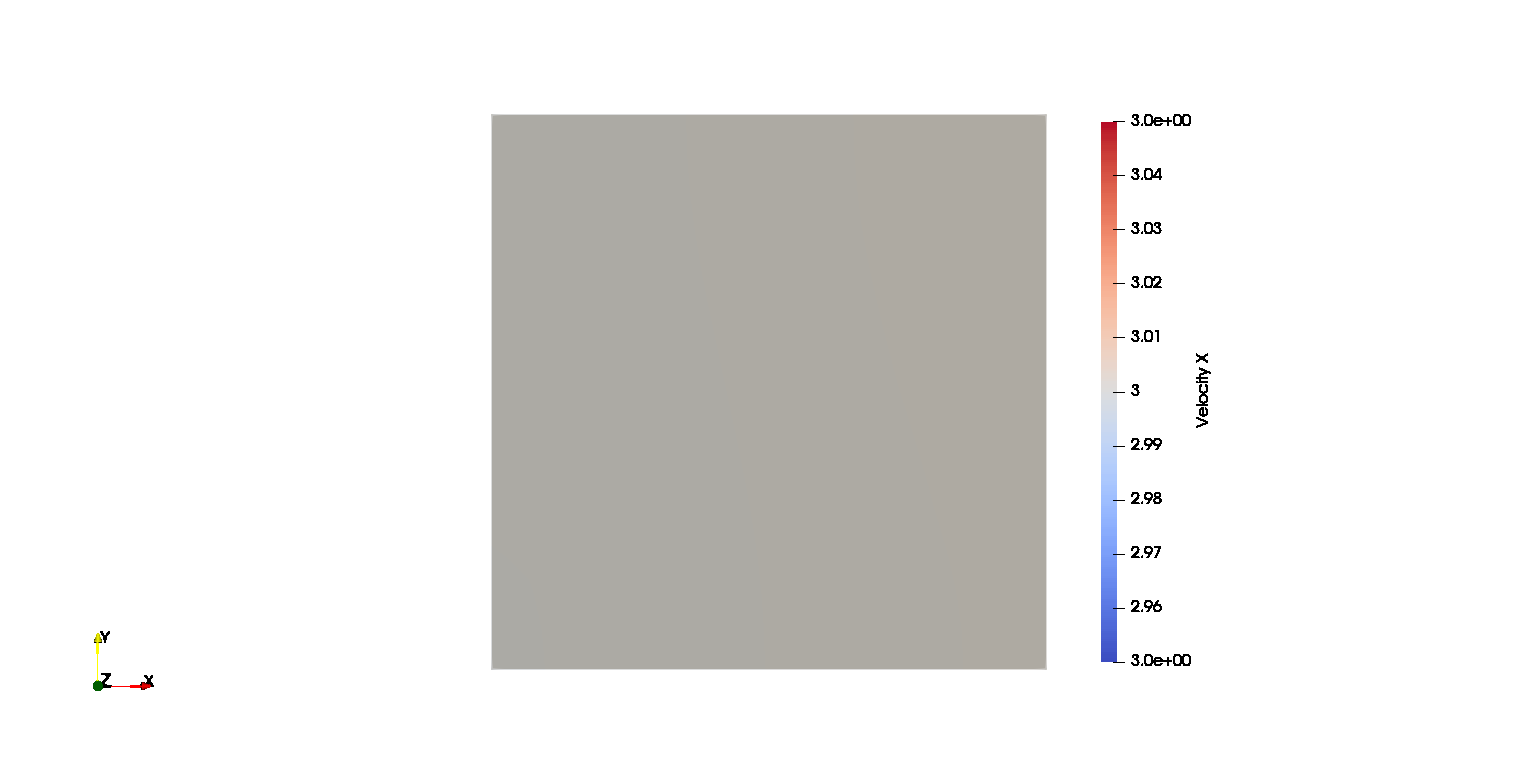}
         \caption{$u_1$}
     \end{subfigure}
     \quad \quad \quad \;
     \begin{subfigure}[b]{0.45\textwidth}
         \centering
         \includegraphics[scale=0.18]{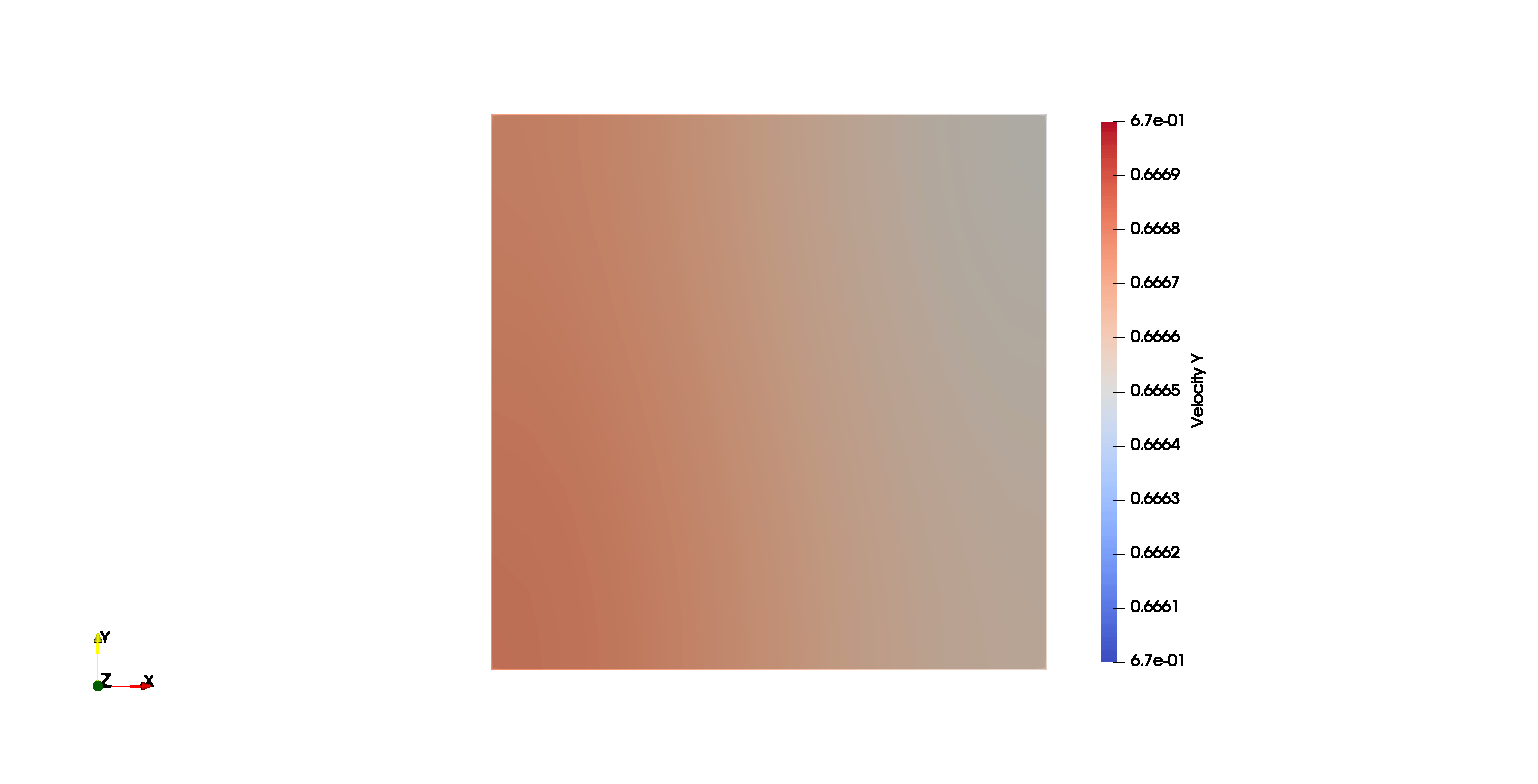}
         \caption{$u_2$}
     \end{subfigure}
     \quad \quad \;
        \caption{NIRB approximations ($u_1$ (left) and $u_2$ (right)) for $T=5$ with $N=10$ modes (close to $(a,\frac{b}{a})=(3,2/3)$) (note the small scale that is used)}
        \label{fig:solNIRB}
   \end{figure}

  \subsubsection{ Time execution (min,sec)}
\label{time3}
Finally, the computational costs are well saved during the online part of the algorithm as it is highlited with this example. Indeed, since there is a nonlinearity, the system must be solved with several iterations for each time step, and thus is quite expensive for a HF approximation. We recall that with an explicit Euler scheme, the solution blows up whereas for a explicit RK2 scheme (without iteration), the solution converges to the expected values $(a,\frac{b}{a})$.

We present the FEM and NIRB runtimes in \ref{runtimes13} and \ref{runtimes23}.
   \begin{table}[tbhp]
     {\footnotesize
       \caption{ FEM runtimes (min:sec)}\label{runtimes13}
   \begin{center}
    \begin{tabular}{ |c| c| }
      \hline
   FEM high fidelity solver & FEM coarse solution \\
    \hline
    4:52 & 00:02\\
    \hline
    \end{tabular}
   \end{center}
   }
   \end{table}
   \begin{table}[tbhp]
     {\footnotesize
       \caption{NIRB runtimes ($N=10$, h:min:sec)}\label{runtimes23}
   \begin{center}
    \begin{tabular}{|c|c|}
      \hline
      NIRB Offline & classical rectified NIRB online\\
      \hline
      1:53:00 & 00:04:00\\
      \hline
    \end{tabular}
   \end{center}
   }
   \end{table}
\bibliography{parabolic}
\bibliographystyle{plain}
\end{document}